\documentclass[12pt]{article}
\usepackage{amssymb,amsmath}
\usepackage[numbers,sort&compress]{natbib}
\def\a{\alpha}
\def\b{\beta}
\def\d{\delta}

\def\l{\lambda}

\def\ssc{\scriptscriptstyle}

\def\cl{\centerline}

\def\vs{\vspace*}

\def\Z{\mathbb{Z}{\ssc\,}}

\def\C{\mathbb{C}{\ssc\,}}

\numberwithin{equation}{section}
\newtheorem{theo}{Theorem}[section]
\newtheorem{coro}[theo]{Corollary}
\newtheorem{defi}[theo]{Definition}
\newtheorem{rema}[theo]{Remark}
\newtheorem{lemm}[theo]{Lemma}
\newtheorem{prop}[theo]{Proposition}

\def\adddot{$\!\!\!${\bf.}\ \ }

\begin{document}
\cl{{\Large \bf
Simple Restricted modules for $N=2$ Neveu-Schwarz algebra}} \vs{6pt}

\cl{Yuezhu Wu, Wei Jiang, Yan He}

 \cl{\small School of Mathematics and statistics, Suzhou University of Technology, Jiangsu 215500, China}

 \cl{\small E-mails: yuezhuwu@szut.edu.cn, \ jiangwei@szut.edu.cn, \ 202000001@szut.edu.cn} \vs{6pt}

 \vs{6pt}

{\small
\parskip .005 truein
\baselineskip 3pt \lineskip 3pt

\noindent{{\bf Abstract.} In this paper, we first construct many simple restricted modules for the $N=2$ Neveu-Schwarz algebra $\mathcal{L}$. These modules contain the highest weight modules and Whittaker modules. Then we   precisely characterize the simple restricted modules over $\mathcal{L}$ under certain conditions,  and establish several equivalent conditions for a simple \(\mathcal{L}\)-module to be restricted. Finally, we prove that if one specific positive root element of $\mathcal{L}$ acts locally finitely on a simple $\mathcal{L}$-module $V$, then $V$  must be   restricted.

\vs{5pt}

\noindent{\bf Key words:} $N=2$ Neveu-Schwarz algebra, Restricted modules, Simple modules.}

\noindent{\it Mathematics Subject Classification (2000):} 17B65,
17B68, 17B70.}
\parskip .001 truein\baselineskip 8pt \lineskip 8pt

\vs{6pt}

\section{Introduction}

Super-Virasoro algebras, i.e., superconformal algebras may be viewed as a nontrivial $\Z_2$-graded extension of the
Virasoro algebra and their roots in theoretical physics can be traced in the 1970's \cite{ABdN}. Virasoro algebras and super-Virasoro algebras supply the underlying symmetries of string theory. These strings seem to be related to M-theory.
The $N=2$ superconformal algebras, constructed independently by Kac \cite{K1} and by Ademollo {\it et
al.}\cite{ABdN}, whose structure and representation theories were investigated in a series of papers (e.g., Refs.
\cite{CDLP1,CDLP2,CSX, DG,FJS,FSST,FST,Io,IK,Kir,KS,KV,LPX, LSZ,ST,YS,YYX}, etc.), have attracted more and more attentions in recent years. By choosing the
standard Virasoro generators, one observes that the  $N=2$ superconformal algebras fall into the following
four types: (i) the Neveu-Schwarz algebra resulting from antiperiodic boundary conditions, (ii) the
Ramond algebra corresponding to periodic boundary conditions, (iii) the topological $N=2$ algebra
being as the symmetry algebra of topological conformal field theory, and (iv) the twisted superconformal algebra possessing mixed boundary conditions. It is known that the first three superconformal algebras are isomorphic to each other.

It is well known that highest weight modules and Whittaker modules are most popular
weight modules and non-weight modules respectively, for many infinite-dimensional Lie (super)algebras such as Virasoro algebras, super-Virasoro algebras and other related (super)algebras.
Since highest weight modules and Whittaker modules have a same property: the actions of elements in the positive part of the (super)algebra are locally finite, thus it is natural to
investigate such class of modules in a uniform way.  Such modules are referred to as restricted modules (or smooth modules), which have been investigated for the Virasoro algebra \cite{MZ2,MNTZ},
 the Heisenberg-Virasoro algebras
 \cite{CG, G, GXZ}, affine Virasoro algebra \cite{GXZ, MNTZ}, Neveu-Schwarz algebra \cite{LPX20}, twisted $N=2$ superconformal algebra \cite{CSX}, etc.
In the present paper, we study the simple restricted modules for the $N=2$ Neveu-Schwarz algebra.

The paper is organized as follows. In Section 2, we recall some notations and collect
some known results related to the $N=2$ Neveu-Schwarz algebra $\mathcal{L}$ for later use. In Section
3, we use simple modules over some subalgebras of $\mathcal{L}$ to construct simple restricted $\mathcal{L}$-modules. These modules contain the highest weight modules and Whittaker modules. In Section 4, we give  precise characterizations
of the simple restricted modules over $\mathcal{L}$ under certain conditions (see Theorem \ref{MainTh} and Theorem \ref{local}), and  give several equivalent conditions for a simple $\mathcal{L}$-module to be  restricted.  Finally, we  prove that if $V$ is a simple $\mathcal{L}$-module and $L_k$ acts locally finitely on  $V$ for a fixed positive integer $k$, then $V$  must be  a restricted module  (see Theorem \ref{last}).


Throughout this paper, denote by $\Z$ and $\C$ the sets of integers and complex numbers, respectively. Denote by $\Z_{\ge i}$ (resp. $\Z_{>i}$) the set of all integers $k$ such that $k\ge i$ (resp. $k>i$). For a Lie superalgebra $\mathfrak{a}$, denote by $U(\mathfrak{a})$  the universal enveloping algebra of $\mathfrak{a}$. All vector spaces and algebras are over $\C$.

\section{Preliminaries}
In this section, we recall some definitions and notations for later use.
\begin{defi}\adddot
 The $N=2$ Neveu-Schwarz algebra $\mathcal{L}:=\mathcal{L}_{\bar 0}\oplus \mathcal{L}_{\bar 1}$ is the infinite-dimensional Lie superalgebra
 whose even part is spanned by $\{L_m,T_n, C\,|\,m,n\in\Z\}$,  whose odd part is spanned by
 $\{G_{r+\frac{1}{2}}, Q_{s+\frac{1}{2}}\,|\,r,s\in\Z\}$, and  whose defining relations  are given by
\begin{eqnarray*}\label{NS}
&&[L_m,L_n]=(m-n)L_{m+n}+\frac{C}{12}(m^3-m)\d_{m+n,0},\\
&&[L_m,T_n]=-nT_{m+n},\\
&&[T_m,T_n]=\frac{C}{3}m\d_{m+n,0},\\
&&[L_m,G_{r+\frac{1}{2}}]=(\frac{m}{2}-r-\frac{1}{2})G_{m+r+\frac{1}{2}},\quad [L_m,Q_{r+\frac{1}{2}}]=(\frac{m}{2}-r-\frac{1}{2})Q_{m+r+\frac{1}{2}},\\
&&[T_m,G_{r+\frac{1}{2}}]=G_{m+r+\frac{1}{2}},\quad [T_m,Q_{r+\frac{1}{2}}]=-Q_{m+r+\frac{1}{2}},\\
&&[G_{m+\frac{1}{2}},Q_{n+\frac{1}{2}}]=2L_{m+n+1}+(m-n)T_{m+n+1}+\frac{C}{3}\left((m+\frac{1}{2})^2-\frac{1}{4}\right)\d_{m+n+1,0},\\
&&[\mathcal{L},C]=0=[G_{m+\frac{1}{2}},G_{n+\frac{1}{2}}]=[Q_{m+\frac{1}{2}},Q_{n+\frac{1}{2}}].
\end{eqnarray*}
\end{defi}

Obviously, $\mathcal{L}$ is $\frac{1}{2}\Z$-graded, and the Cartan subalgebra  is $\eta: = {span}\{L_0, T_0, C\}$. One can easily
 see that $\mathcal{L}$ contains the Virasro algebra $\mathcal{V}:={span}\{L_i, C \,|\,i\in \Z\}$,
 the Heisenberg algebra $\mathcal{H}={span}\{T_n,C\,|\,n\in \Z\}$, and the Heisenberg-Virasoro algebra $\mathcal{HV}=\mathcal{V}+\mathcal{H}.$
Set $\mathcal{L}^{+}=\sum\limits_{i>0}(\C L_i+\C T_i)+\sum\limits_{j\ge 0}(\C G_{j+\frac{1}{2}}+\C Q_{j+\frac{1}{2}})$, $\mathcal{L}^{-}=\sum\limits_{i<0}(\C L_i+\C T_i+\C G_{i+\frac{1}{2}}+\C Q_{i+\frac{1}{2}})$. Then we have the triangular decomposition:
$$\mathcal{L}=\mathcal{L}^- \oplus \eta \oplus \mathcal{L}^+.$$

 \begin{defi} Let $\mathfrak{g}=\oplus_{i\in \frac{1}{2}\Z}\mathfrak{g}_i$ be a Lie superalgebra. A $\mathfrak{g}$-module $V$ is called a
 {\bf restricted module} if  for any $v\in V$ there exists $n\in\Z_{\ge 1}$ such that $\mathfrak{g}_iv=0$ for all  $i>n$.
 \end{defi}

\begin{defi} Let $\mathfrak{g}$ be a Lie superalgebra and $V$ a $\mathfrak{g}$-module. Let $x\in \mathfrak{g}$,
if for any $v\in V$ there exists $n\in \Z_{\ge 1}$ such that $x^nv=0$, then we say that  $x$ acts
 {\bf  locally nilpotently}  on $V$. We say that $\mathfrak{g}$ acts {\bf locally nilpotently} on $V$
if for any $v\in V$ there exists $n\in \Z_{\ge 1}$ such that $x_1x_2\cdots x_nv=0$ for any $x_1,x_2,\cdots, x_n\in \mathfrak{g}$.
\end{defi}

\begin{defi} Let $\mathfrak{g}$ be a Lie superalgebra and $V$ a $\mathfrak{g}$-module. Let $x\in \mathfrak{g}$,
if for any $v\in V$ we have ${ dim}(\sum_{n\in\Z_{\ge 0}}\C x^nv)<\infty$,  then we say that  $x$ acts
{\bf  locally finitely} on $V$. We say  that  $\mathfrak{g}$  acts {\bf locally finitely} on $V$
if for any $v\in V$  there exists a finite-dimensional $\mathfrak{g}$-submodule $W$ of $V$ such that $v\in W$.
\end{defi}

\begin{rema}It is clearly that $x$ acts locally nilpotently on $V$ is stronger than $x$ acts locally finitely on $V$. If $\mathfrak{g}$ is a finitely generated Lie superalgebra, then
 $\mathfrak{g}$ acts locally nilpotently on $V$ implies that $\mathfrak{g}$ acts locally finitely on $V$.
 But it is not true in general. Clearly, the $N=2$ Neveu-Schwarz algebra $\mathcal{L}$ is a finitely generated Lie superalgebra.
\end{rema}

For latter use, we introduce some subalgebras of $\mathcal{L}$. Let
$$\mathfrak{b}=\sum\limits_{i\in\Z_{\ge 0}}(\C L_i+\C T_{i}+\C G_{i+\frac{1}{2}}+\C Q_{i+\frac{1}{2}})+\C C.$$
For $d\in \Z_{\ge 0}$, $p,q\in\Z$ with $p+q+3\ge  d$, let
$$\mathcal{L}_{d,p,q}=\sum\limits_{i\in\Z_{\ge 0}}(\C L_i+\C T_{i+d}+\C G_{i+p+\frac{3}{2}}+\C Q_{i+q+\frac{3}{2}})+\C T_0+\C C.$$
When $d=0$, denote $\mathcal{L}_{d,p,q}$ by $\mathcal{L}_{0,p,q}$.
Clearly, $\mathfrak{b}$, $\mathcal{L}_{d,p,q}$ and $\mathcal{L}_{0,p,q}$  are  subalgebras of $\mathcal{L}$.

Let $\mathfrak{a}\in\{\mathfrak{b},\mathcal{L}_{d,p,q},\mathcal{L}_{0,p,q}\}$ be fixed. For any $\mathfrak{a}$-module $V$ and scalar $\ell\in\C$,  define the the  induced module
$Ind(V)=U(\mathcal{L})\otimes_{U(\mathfrak{a})}V,$ and let
$$Ind_{\ell}(V):=Ind(V)/(C-\ell)Ind(V).$$

Let  $\mathbb{S}_0$  be  the set of all (infinite) vectors of the form ${\bf i}:=(\ldots, i_2,i_1)$ with entries in $\Z_{\ge 0}$ such that the number of nonzero entries is finite. Let
$\mathbb{S}_1:=\{{\bf{i}}\in \mathbb{S}_0\,|\,i_k=0,1, \ k\in \Z_{\ge 1}\}$.
Let ${\bf 0}$ denote the element $(\ldots,0,0)\in \mathbb{S}_0$. For $i\in\Z_{\ge 1}$, let $\epsilon_i$ denote the element $(\ldots,0,1,0,\ldots,0)\in\mathbb{S}_0$, where $1$ is
in the $i$th position from the right.
For any ${\bf i, j}\in \mathbb{S}_0$ and ${\bf m,n}\in \mathbb{S}_1$, we denote
\begin{eqnarray*}
&&Q^{\bf n}G^{\bf m}T^{\bf j}L^{\bf i}:=\cdots Q_{-2+\frac{1}{2}}^{n_{2}} Q_{-1+\frac{1}{2}}^{n_{1}}\cdots G_{-2+\frac{1}{2}}^{m_{2}} G_{-1+\frac{1}{2}}^{m_{1}}
\cdots T_{-2}^{j_{2}} T_{-1}^{j_{1}}\cdots L_{-2}^{i_{2}} L_{-1}^{i_{1}}\in U({\mathcal{L}}),\\
&&{\bf w(n,m,j,i)}=\sum\limits_{t\in\Z_{\ge 1}}n_t\cdot (t-\frac{1}{2})+\sum\limits_{t\in\Z_{\ge 1}}m_t\cdot (t-\frac{1}{2})+\sum\limits_{t\in\Z_{\ge 1}}t\cdot j_t+\sum\limits_{t\in\Z_{\ge 1}}t\cdot i_t,\\
&& {\bf d(n,m,j,i)}=\sum\limits_{t\in\Z_{\ge 1}} (n_t+m_t+j_t+i_t).
\end{eqnarray*}
Clearly, ${\bf w(n,m,j,i)}$ and ${\bf d(n,m,j,i)}$ are  nonnegative integers.
${\bf w(n,m,j,i)}$ is called the length of ${\bf (n,m,j,i)}$ (or the length of $ Q^{\bf n}G^{\bf m}T^{\bf j}L^{\bf i}$, denoted by ${\bf w}\left(Q^{\bf n}G^{\bf m}T^{\bf j}L^{\bf i}\right)$);  ${\bf d(n,m,j,i)}$ is called the degree of  ${\bf (n,m,j,i)}$ (or the degree of $Q^{\bf n}G^{\bf m}T^{\bf j}L^{\bf i}$, denoted by  ${\bf d}\left(Q^{\bf n}G^{\bf m}T^{\bf j}L^{\bf i}\right)$).

Denote by $<$ the reverse lexicographic total order on $\mathbb{S}_0$ (resp. $\mathbb{S}_1$), which is defined recursively  as follows (see \cite{MZ1,LPX20,G}): $\bf 0$ is the minimum element; and for any different nonzero ${\bf i,j}\in \mathbb{S}_0$ (resp. $\mathbb{S}_1)$, ${\bf i}<{\bf j}$ if and only if the following conditions are satisfied:

(1) $\min\{s\,|\,i_s\ne 0\}>\min\{s\,|\,j_s\ne 0\}$;

(2) $\min\{s\,|\,i_s\ne 0\}=\min\{s\,|\,j_s\ne 0\}=k$ and ${\bf i}-\epsilon_k<{\bf j}-\epsilon_k.$

Similarly, one can  define the reverse lexicographic total order on $\Z^m_{\ge 0}$ and $\Z_2^m$, where $\Z_2=\{0,1\}$ and $ m\in\Z_{\ge 1}.$

Now define the principal total order $\succ $ on $\mathbb{S}_1\times \mathbb{S}_1\times\mathbb{S}_0\times \mathbb{S}_0$ as follows: for different ${\bf (n,m,j,i)},{\bf (n',m',j',i')}\in \mathbb{S}_1\times \mathbb{S}_1\times\mathbb{S}_0\times \mathbb{S}_0$, ${\bf (n,m,j,i)}\succ {\bf (n',m',j',i')}$ if and only if the following conditions are satisfied:

(1) $ {\bf w(n,m,j,i)}>{\bf w(n',m',j',i')}$;

(2) $ {\bf w(n,m,j,i)} = {\bf w(n',m',j',i')}$ and ${\bf d(n,m,j,i)}>{\bf d(n',m',j',i')}$;

(3) $ {\bf w(n,m,j,i)} = {\bf w(n',m',j',i')},$ ${\bf d(n,m,j,i)}={\bf d(n',m',j',i')}$, and ${\bf i}>{\bf i'}$;

(4)  $ {\bf w(n,m,j,i)} = {\bf w(n',m',j',i')},$ ${\bf d(n,m,j,i)}={\bf d(n',m',j',i')}$, ${\bf i}={\bf i'}$, and ${\bf j}>{\bf j'}$;

(5) $ {\bf w(n,m,j,i)} = {\bf w(n',m',j',i')},$  ${\bf d(n,m,j,i)}={\bf d(n',m',j',i')}$, ${\bf i}={\bf i'}$, ${\bf j}={\bf j'}$, and ${\bf m}>{\bf m'}$;

(6) $ {\bf w(n,m,j,i)} = {\bf w(n',m',j',i')},$ ${\bf d(n,m,j,i)}={\bf d(n',m',j',i')}$, ${\bf i}={\bf i'}$, ${\bf j}={\bf j'}$, ${\bf m}={\bf m'}$, and ${\bf n}>{\bf n'}$.

Similarly, one can define the principal total order on $\Z_2^m\times\Z_2^n$, where $m,n\in\Z_{\ge 1}$.

\section{Construction of simple restricted modules}
 In this section, we will construct some simple restricted modules for the $N=2$  Neveu-Schwarz algebra $\mathcal{L}$.
\begin{lemm}\label{Tb-irr}
Let $S$ be a simple $\mathfrak{b}$-module. Assume that there exists $s\in\Z_{\ge 1}$ such that $T_s$ acts injectively on $S$, and $L_iS=T_jS=G_{i+\frac{1}{2}}S=Q_{i+\frac{1}{2}}S=0$ for all $i\ge s,j>s$.
Then for all $\ell\in\C$, $Ind_\ell(S)$ is a simple $\mathcal{L}$-module.
\end{lemm}

{\it Proof.}   Let $\succ$ be the principal  total order on $\mathbb{S}_1\times \mathbb{S}_1\times \mathbb{S}_0\times \mathbb{S}_0$,  which is introduced in Section 2. By the PBW Theorem, each $v\in Ind_\ell(S)$ can be uniquely written in  the form
\begin{equation}\label{PBWform1}
\sum\limits_{({\bf n',m',j',i'})\in \mathbb{S}_1\times \mathbb{S}_1\times \mathbb{S}_0\times \mathbb{S}_0}Q^{\bf n'}G^{\bf m'}T^{\bf j'}L^{\bf i'}v_{\bf n',m',j',i'},
\end{equation}
 where $v_{{\bf n',m',j',i'}}\in S,$ and $$Q^{{\bf n'}}=\cdots Q_{-2+\frac{1}{2}}^{n'_{2}} Q_{-1+\frac{1}{2}}^{n'_{1}}, \quad
 G^{{\bf m'}}=\cdots G_{-2+\frac{1}{2}}^{m'_{2}} G_{-1+\frac{1}{2}}^{m'_{1}},\quad
 T^{{\bf j'}}=\cdots T_{-2}^{j'_{2}} T_{-1}^{j'_{1}},\quad
 L^{{\bf i'}}=\cdots L_{-2}^{i'_{2}} L_{-1}^{i'_{1}}.$$

For any $v\in Ind_\ell(S)$, let $supp(v)$ be the set of all $({\bf n',m',j',i'})$ with $v_{{\bf n',m',j',i'}}\ne 0$, and $deg(v)$ be the maximal element of $supp(v)$ with respect to the  order $\succ$.  Let $deg(v)=({\bf n,m,j,i})$, and call $Q^{\bf n}G^{\bf m}T^{\bf j}L^{\bf i}v_{\bf n,m,j,i}$ the leading term of $v$.

In order to prove the simplicity of $Ind_\ell(S)$, we need the following Claim.

{\bf Claim 1.} (1) If ${\bf i}\ne 0$, let $a:=\min\{s\,|\,i_s\ne 0\}.$ Then $deg(T_{s+a}v)=({\bf n,m,j,i}-\epsilon_a).$

(2) If $ {\bf j}\ne 0$, let $b:=\min\{s\,|\,j_s\ne 0\}.$ Then  $deg(L_{s+b}v)=({\bf n,m,j}-\epsilon_b,{\bf i}).$

(3) If ${\bf m}\ne 0$, let $c:=\min\{s\,|\,m_s\ne 0\}.$ Then $deg(Q_{s+c-\frac{1}{2}}v)=({\bf n,m}-\epsilon_c,{\bf j,i}).$

(4) If $ {\bf n}\ne 0$, let $d:=\min\{s\,|\,n_s\ne 0\}.$ Then $deg(G_{s+d-\frac{1}{2}}v)=({\bf n}-\epsilon_d,{\bf m,j,i}).$

We prove Claim 1 (1) as an example, as the others can be proved similarly. Assume  ${\bf i}\ne 0$ and  let $a:=\min\{s\,|\,i_s\ne 0\}.$
 Note that $T_{s+a}v_{\bf n',m',j',i'}=0$ for all $({\bf n',m',j',i'})\in supp(v).$ Then

\begin{eqnarray*}&&T_{s+a}(Q^{\bf n'}G^{\bf m'}T^{\bf j'}L^{\bf i'}v_{\bf n',m',j',i'})\\
&=&\left([T_{s+a},Q^{\bf n'}]G^{\bf m'}T^{\bf j'}L^{\bf i'}+Q^{\bf n'}[T_{s+a},G^{\bf m'}]T^{\bf j'}L^{\bf i'}\right)v_{\bf n',m',j',i'}\\
&&+\left(Q^{\bf n'}G^{\bf m'}[T_{s+a},T^{\bf j'}]L^{\bf i'}+Q^{\bf n'}G^{\bf m'}T^{\bf j'}[T_{s+a},L^{\bf i'}]\right)v_{\bf n',m',j',i'}.
\end{eqnarray*}
Write this as the sum of some elements written in the form \eqref{PBWform1}. Then one can see that all
the lengths of summands coming from

$$[T_{s+a},Q^{\bf n'}]G^{\bf m'}T^{\bf j'}L^{\bf i'}v_{\bf n',m',j',i'},\ Q^{\bf n'}[T_{s+a},G^{\bf m'}]T^{\bf j'}L^{\bf i'}v_{\bf n',m',j',i'},\ Q^{\bf n'}G^{\bf m'}[T_{s+a},T^{\bf j'}]L^{\bf i'}v_{\bf n',m',j',i'}$$ are strictly smaller than ${\bf w(n',m',j',i')}-s$.

Now consider the summands coming from $Q^{\bf n'}G^{\bf m'}T^{\bf j'}[T_{s+a},L^{\bf i'}]v_{\bf n',m',j',i'}$. Let  $a':=\min\{s\,|\,i_s'\ne 0\}.$   If $a<a'$, all the lengths of summands coming from
$$Q^{\bf n'}G^{\bf m'}T^{\bf j'}[T_{s+a},L^{\bf i'}]v_{\bf n',m',j',i'}$$ are strictly smaller than ${\bf w(n',m',j',i')}-s.$

If  $a>a'$, all the lengths of summands coming from
$$Q^{\bf n'}G^{\bf m'}T^{\bf j'}[T_{s+a},L^{\bf i'}]v_{\bf n',m',j',i'}$$ are  smaller than ${\bf w(n',m',j',i')}-s,$ and there may be some summands whose lengths are equal to ${\bf w(n',m',j',i')}-s.$ But the degrees of these summands, whose lengths are equal to ${\bf w(n',m',j',i')}-s$, must be smaller than   ${\bf d(n',m',j',i')}-2$.

If $a=a'$,  all the lengths of summands coming from
$$Q^{\bf n'}G^{\bf m'}T^{\bf j'}[T_{s+a},L^{\bf i'}]v_{\bf n',m',j',i'}$$ are  smaller than ${\bf w(n',m',j',i')}-s,$ and there exists a  summand whose length is equal to ${\bf w(n',m',j',i')}-s,$ and whose degree is equal to
${\bf d(n',m',j',i')}-1$. Furthermore, one can see that, up to a  nonzero scalar,    the leading term of $Q^{\bf n'}G^{\bf m'}T^{\bf j'}[T_{s+a},L^{\bf i'}]v_{\bf n',m',j',i'}$ is equal to
$$Q^{\bf n'}G^{\bf m'}T^{\bf j'}L^{{\bf i'}-\epsilon_a}v_{\bf n',m',j',i'}.$$

From the above argument, we obtain that, up to a  nonzero scalar,  the leading term of $T_{s+a}v$ is equal to
$$Q^{\bf n}G^{\bf m}T^{\bf j}L^{{\bf i}-\epsilon_a}v_{\bf n,m,j,i},$$ and $deg(T_{s+a}v)=({\bf n,m,j,i}-\epsilon_a).$ This proves  Claim 1 (1).

By applying Claim 1 repeatedly,  one can obtain a nonzero element in $U(\mathcal{L})v\cap S$ from any nonzero element $v\in Ind_\ell(S)$.  This proves the simplicity of $Ind_\ell(S)$.

\begin{lemm}\label{Lb-irr} Let   $S$ be a simple $\mathfrak{b}$-module. Assume that there exists $r\in\Z_{\ge 1}$ such that $L_r$ acts injectively on $S$, and  $L_iS=T_jS=G_{j+\frac{1}{2}}S=Q_{j+\frac{1}{2}}S=0$ for all $i>r,j\ge r$.
Then for all $\ell\in\C$, $Ind_\ell(S)$ is a simple $\mathcal{L}$-module.
\end{lemm}

{\it Proof.}   Let $\succ$ be the principal  total order on $\mathbb{S}_1\times \mathbb{S}_1\times \mathbb{S}_0\times \mathbb{S}_0$, which is introduced in Section 2. By the PBW Theorem, each $v\in Ind_\ell(S)$ can be uniquely written in  the form
\begin{equation}\label{PBWform}
\sum\limits_{({\bf n',m',j',i'})\in \mathbb{S}_1\times \mathbb{S}_1\times \mathbb{S}_0\times \mathbb{S}_0}Q^{\bf n'}G^{\bf m'}T^{\bf j'}L^{\bf i'}v_{\bf n',m',j',i'},
\end{equation}
 where $v_{{\bf n',m',j',i'}}\in S,$ and $$Q^{{\bf n'}}=\cdots Q_{-2+\frac{1}{2}}^{n'_{2}} Q_{-1+\frac{1}{2}}^{n'_{1}}, \quad
 G^{{\bf m'}}=\cdots G_{-2+\frac{1}{2}}^{m'_{2}} G_{-1+\frac{1}{2}}^{m'_{1}},\quad
 T^{{\bf j'}}=\cdots T_{-2}^{j'_{2}} T_{-1}^{j'_{1}},\quad
 L^{{\bf i'}}=\cdots L_{-2}^{i'_{2}} L_{-1}^{i'_{1}}.$$

For any $v\in Ind_\ell(S)$, let $supp(v)$ be the set of all $({\bf n',m',j',i'})$ with $v_{{\bf n',m',j',i'}}\ne 0$, and $deg(v)$ be the maximal element of $supp(v)$ with respect to the above  total order $\succ$.  Let $deg(v)=({\bf n,m,j,i})$.

 Similarly to Claim 1, we have the following Claim.

{\bf Claim 2.} (1) If ${\bf i}\ne 0$, let $a:=\min\{s\,|\,i_s\ne 0\}.$ Then $deg(L_{r+a}v)=({\bf n,m,j,i}-\epsilon_a).$

(2) If ${\bf m}\ne 0$, let $b:=\min\{s\,|\,m_s\ne 0\}.$ Then $deg(Q_{r+b-\frac{1}{2}}v)=({\bf n,m}-\epsilon_b,{\bf j,i}).$

(3) If ${\bf n}\ne 0$, let $c:=\min\{s\,|\,n_s\ne 0\}.$ Then $deg(G_{r+c-\frac{1}{2}}v)=({\bf n}-\epsilon_c,{\bf m,j,i}).$

Now we consider the remaining case ${\bf i}={\bf m}={\bf n}=0,{\bf j}\ne 0$. Let $d:=\min\{s\,|\,j_s\ne 0\}.$
If there exists $0\ne w\in S$   such that
$  G_{k+\frac{1}{2}}w= Q_{k+\frac{1}{2}}w=0, k\ge 0.$
One can choose $k,l\ge 0$ with $k+l+1=r$. Then
$$[G_{k+\frac{1}{2}}, Q_{l+\frac{1}{2}}]w=(2L_r+(k-l)T_r)w=2L_rw=0,$$
which contradicts the injectivity of $L_r$.
Therefore, for $v_{\bf n,m,j,i}$, there exists $p\in\Z_{\ge 0}$  such that
\begin{equation}\label{G} G_{p+\frac{1}{2}}v_{\bf n,m,j,i}\ne 0,\ G_{k+\frac{1}{2}}v_{\bf n,m,j,i}=0, k>p,\end{equation} or
\begin{equation} Q_{p+\frac{1}{2}}v_{\bf n,m,j,i}\ne 0,\ Q_{k+\frac{1}{2}}v_{\bf n,m,j,i}=0, k>p.\end{equation}
Without loss of  generality, we assume that \eqref{G} holds. Now, one can see that $deg(G_{d+p+\frac{1}{2}} v)=({\bf n,m},{\bf j}-\epsilon_d,{\bf i})$.

By applying this and  Claim 2 repeatedly,  one can obtain a nonzero element in $U(\mathcal{L})v\cap S$ from any nonzero element $v\in Ind_\ell(S)$. This proves the simplicity of $Ind_\ell(S)$.

By an analogous argument,  we arrive at  the following Lemma.
\begin{lemm}\label{LTb-irr} Let   $S$ be a simple $\mathfrak{b}$-module. Assume that there exists $r\in\Z_{\ge 1}$ such that $L_r,$ $T_r$, $L_r+(-2a-r+1)T_r, L_r+(r+2b-1)T_r,\, \forall a,b\in\Z_{\ge 1}$ act injectively on $S$, and  $L_iS=T_iS=G_{j+\frac{1}{2}}S=Q_{j+\frac{1}{2}}S=0$ for all $i>r,j\ge r$. Then for all $\ell\in\C$, $Ind_\ell(S)$ is a simple $\mathcal{L}$-module.
\end{lemm}

\begin{rema}\label{no-irr} If we drop the simplicity of $S$ in Lemmas \ref{Tb-irr}-\ref{LTb-irr}, one  may still obtain a nonzero element in $U(\mathcal{L})v\cap S$ from any nonzero element $v\in Ind_\ell(S)$.
\end{rema}

\begin{theo}\label{T-th} Let $V$ be a simple $\mathcal{L}_{d,p,q}$-module. Assume that there exist $r,s\in\Z_{\ge 0},\, p,q\in\Z$ with $\max\{p,q\}<\max\{r,s\}$ such that $L_r,T_s$ act injectively on $V$, and $L_iV=T_jV=G_{k+\frac{1}{2}}V=Q_{l+\frac{1}{2}}V=0$ for all $i>r,j>s,k>p,l>q$.
If $0\le r<s,\ p+q+3=s,s+1$ and $d=s-r$,  the following results hold.

{\rm (1)}  Set $M:=Ind^{\mathcal{L}_{0,p,q}}_{\mathcal{L}_{d,p,q}}(V)$.  Then $M$ is a simple $\mathcal{L}_{0,p,q}$-module. Furthermore,
\begin{equation}\label{M}\left\{\begin{array}{l} T_s \mbox{ acts injectively on } M,\\
T_j(M)=0,j>s, \quad  L_j(M)=0,j>s-1,\\
 G_{k+\frac{1}{2}}(M)=0,k>p,\quad  Q_{l+\frac{1}{2}}(M)=0, l> q.\end{array}\right.
\end{equation}

{\rm (2)}  For all $\ell\in\C$, $Ind_\ell(V)$ is a simple $\mathcal{L}$-module.
\end{theo}
{\it Proof.}
(1) Let $\succ$ be the reverse lexicographic total order on $\Z_{\ge 0}^{d-1}$. For any $v\in M$, by the PBW Theorem we can write $v$ in the form
$\sum_{{\bf k}\in\Z_{\ge 0}^{d-1}}T^{\bf k}v_{\bf k}$, where $v_{\bf k}\in V$ and
$$T^{\bf k}=T^{k_{d-1}}_{d-1}T^{k_{d-2}}_{d-2}\cdots T^{k_{1}}_{1}.$$

 For any $v\in M\setminus V$, let $supp(v)$ be the set of all $\bf k$ with $v_{\bf k}\ne 0$, and $deg(v)$ be the maximal element of $supp(v)$ with respect to the above   order $\succ$.  Let $deg(v)=\bf j$ and $a=\min\{s\,|\,j_s\ne 0\}$.

We aim to show  that
\begin{equation}\label{T-irr}deg(L_{s-a}v)={\bf j}-\epsilon_a,\end{equation} from which one may readily deduce the simplicity of   $M$.

Since $1\le a<d$ and $d=s-r$, we have $s-a>r$ and $L_{s-a}v_{\bf k}=0, \forall \, {\bf k}\in supp(v)$. For any ${\bf k}\in supp(v)$, if $k_a\ne 0$, then
$$deg(L_{s-a}T^{\bf k}v_{\bf k})={\bf k}-\epsilon_a\preceq {\bf j}-\epsilon_a,$$
where the equality holds if and only if ${\bf k}={\bf j}$. If $k_a=0$, then $k_1=\cdots =k_a=0$ by the definition of the order $\succ$. Thus one may readily verify that $L_{s-a}T^{\bf k}v_{\bf k}=0.$  Hence, \eqref{T-irr} holds, and $M$ is a simple $\mathcal{L}_{0,p,q}$-module.

By straightforward observations, one may verify that \eqref{M} holds.

(2) The proof for   the simplicity of $Ind_\ell(V)$ will be divided into four cases: (i) $p\ge 0,q\ge 0$; (ii) $p\ge 0,q< 0$; (iii) $p\le 0,q\ge 0$; (iv) $p<0,q<0$. We prove   (i) $p\ge 0,q\ge 0$ as an example, as the others can be proved similarly.

 Let $\succ$ be the principal total order on $\Z_2^{q+1}\times \Z_2^{p+1}$  introduced in Section 2. Set
 $$W=\sum\limits_{({\bf l,k})\in \Z_2^{q+1}\times \Z_2^{p+1}}Q^{\bf l}G^{\bf k}v_{\bf l,k},$$
 where $v_{{\bf l,k}}\in M$, $$Q^{{\bf l}}=Q_{q+\frac{1}{2}}^{l_{q}}Q_{q-1+\frac{1}{2}}^{l_{q-1}}\cdots Q_{1+\frac{1}{2}}^{l_{1}} Q_{\frac{1}{2}}^{l_{0}}, \quad
 G^{{\bf k}}=G_{p+\frac{1}{2}}^{k_{p}}G_{p-1+\frac{1}{2}}^{k_{p-1}}\cdots G_{1+\frac{1}{2}}^{k_{1}} G_{\frac{1}{2}}^{k_{0}}.$$
It is clear that $W$ is a $\mathfrak{b}$-module.

 For any $v\in W\setminus M$, let $supp(v)$ be the set of all $(\bf l,k)$ with $v_{\bf l,k}\ne 0$, and $deg(v)$ be the maximal element of $supp(v)$ with respect to the above  total order $\succ$.  Let $deg(v)=(\bf n,m)$.
If ${\bf m}\ne 0$, let $a=\min\{s\,|\,m_s\ne 0\}$. It is follows from $p+q+3=s,s+1$ that  $s-a-\frac{1}{2}\ge q+\frac{3}{2}$. Thus $Q_{s-a-\frac{1}{2}}M=0$ by \eqref{M}. By straight calculations, one can see that $deg(Q_{s-a-\frac{1}{2}}v)=({\bf n,m}-\epsilon_a)$. If ${\bf m}=0$, let $b=\min\{s\,|\,n_s\ne 0\}$. Note that $s-b-\frac{1}{2}\ge p+\frac{3}{2}$. Thus $G_{s-b-\frac{1}{2}}M=0$ by \eqref{M}. One can see that $deg(G_{s-b-\frac{1}{2}}v)=({\bf n}-\epsilon_b,{\bf m})$. Hence,  $W$ is a simple $\mathfrak{b}$-module. Since  $p\ge 0,q\ge 0$, we have $p<s$ and $q<s$.
Then, one can further obtain that
\begin{equation}\label{W-irr}\left\{\begin{array}{l}
 T_s \mbox{ acts injectively on } W,\\
T_j(W)=0,j>s, \quad  L_j(W)=0,j\ge s,\\
 G_{k+\frac{1}{2}}(W)= Q_{k+\frac{1}{2}}(W)=0, k\ge s.\end{array}
\right.
\end{equation}
Now the simplicity of $Ind_\ell(V)$  follows from Lemma \ref{Tb-irr}.

\begin{rema}  To prove the simplicity of $Ind_\ell(V)$ for cases (ii) $p\ge 0,q< 0$; (iii) $p\le 0,q\ge 0$, we  need the condition $\max\{p,q\}<\max\{r,s\}$.
\end{rema}

\begin{theo}\label{L-th} Let $V$ be a simple $\mathcal{L}_{0,p,q}$-module.  Assume that there exist $r,s\in\Z_{\ge 0}, p,q\in\Z$ with  $\max\{p,q\}<\max\{r,s\}$ such that $L_r,T_s$ act injectively on $V$, and  $L_iV=T_jV=G_{k+\frac{1}{2}}V=Q_{l+\frac{1}{2}}V=0$ for all $i>r,j>s,k>p,l>q$.
If $0\le s<r,\ p+q+3=r+1,$ then for all $\ell\in\C$, $Ind_\ell(V)$ is a simple $\mathcal{L}$-module.
\end{theo}
{\it Proof.} The proof is similar to that of Theorem \ref{T-th}. We divid the proof into three cases: (i) $p\ge 0,q\ge 0$; (ii) $p\ge 0,q< 0$; (iii) $p\le 0,q\ge 0$. We prove  (i)  $p\ge 0,q\ge 0$ as an example, as the others can be proved similarly.

Let $\succ$ be the principal total order on $\Z_2^{q+1}\times \Z_2^{p+1}$  introduced in Section 2. Set
 $$W=\sum\limits_{({\bf l,k})\in \Z_2^{q+1}\times \Z_2^{p+1}}Q^{\bf l}G^{\bf k}v_{\bf l,k},$$
 where $v_{{\bf l,k}}\in V$, $$Q^{{\bf l}}=Q_{q+\frac{1}{2}}^{l_{q}}Q_{q-1+\frac{1}{2}}^{l_{q-1}}\cdots Q_{1+\frac{1}{2}}^{l_{1}} Q_{\frac{1}{2}}^{l_{0}}, \quad
 G^{{\bf k}}=G_{p+\frac{1}{2}}^{k_{p}}G_{p-1+\frac{1}{2}}^{k_{p-1}}\cdots G_{1+\frac{1}{2}}^{k_{1}} G_{\frac{1}{2}}^{k_{0}}.$$
It is clear that $W$ is a $\mathfrak{b}$-module.

 For any $v\in W\setminus V$, let $supp(v)$ be the set of all $(\bf l,k)$ with $v_{\bf l,k}\ne 0$, and $deg(v)$ be the maximal element of $supp(v)$ with respect to the above  total order $\succ$.  Let $deg(v)=(\bf n,m)$.
If ${\bf m}\ne 0$, let $a=\min\{s\,|\,m_s\ne 0\}$. It is follows from $p+q+3=r+1$ that  $r-a-\frac{1}{2}\ge q+\frac{3}{2}$. Thus $Q_{r-a-\frac{1}{2}}V=0$. By straight calculations, one can see that $deg(Q_{r-a-\frac{1}{2}}v)=({\bf n,m}-\epsilon_a)$. If ${\bf m}=0$, let $b=\min\{s\,|\,n_s\ne 0\}$. Note that $r-b-\frac{1}{2}\ge p+\frac{3}{2}$. Thus $G_{r-b-\frac{1}{2}}V=0$. One can see that $deg(G_{r-b-\frac{1}{2}}v)=({\bf n}-\epsilon_b,{\bf m})$. Hence,  $W$ is a simple $\mathfrak{b}$-module. Since  $p\ge 0,q\ge 0$, we have $p<r-1$ and $q<r-1$.
Then, one can further see that
\begin{equation}\label{W-irr}\left\{\begin{array}{l}
 L_r \mbox{ acts injectively on } W,\\
L_j(W)=0,j>r, \quad  T_j(W)=0,j\ge r,\\
 G_{k+\frac{1}{2}}(W)= Q_{k+\frac{1}{2}}(W)=0,\ k\ge r.\end{array}
\right.
\end{equation}
Now the simplicity of $Ind_\ell(V)$  follows from Lemma \ref{Lb-irr}.

\begin{theo} \label{LT-th}Let $V$ be a simple $\mathcal{L}_{0,p,q}$-module. Assume that there exist $r\in\Z_{\ge 1}, p,q\in\Z$ with $\max\{p,q\}<\max\{r,s\}$ such that $L_r,T_r$ act injectively on $V$, and  $L_iV=T_iV=G_{k+\frac{1}{2}}V=Q_{l+\frac{1}{2}}V=0$ for all $i>r,k>p,l>q$.
If  $p+q+3=r$, then for all $\ell\in\C$, $Ind_\ell(V)$ is a simple $\mathcal{L}$-module.
\end{theo}

{\it Proof.}  Observe that  $p+q+3=r$ implies that
\begin{equation}\label{p-q}
(2L_r+(p-q)T_r)v=[G_{p+\frac{3}{2}},Q_{q+\frac{3}{2}}]v=0,\quad  \forall v\in V.
\end{equation}
Since $L_r,T_r$ act injectively on $V$, from \eqref{p-q} we have
\begin{equation}\label{=s}
(2L_r+\a T_r)v\ne 0,\quad  \forall \a\ne p-q, 0\ne  v\in V.
\end{equation}

We divide the proof  into the following four cases: (i) $p\ge 0,q\ge 0$; (ii) $p\ge 0,q< 0$; (iii) $p<0,q\ge 0$; (iv) $p<0,q<0$. We prove  (i) $p\ge 0,q\ge 0$ as an example, as the others can be proved similarly.

 Let $\succ$ be the principal  total order on $\Z_2^{q+1}\times \Z_2^{p+1}$ which is introduced in Section 2. Set
 $$W=\sum\limits_{({\bf l,k})\in \Z_2^{q+1}\times \Z_2^{p+1}}Q^{\bf l}G^{\bf k}v_{\bf l,k},$$
 where $v_{{\bf l,k}}\in V$, $$Q^{{\bf l}}=Q_{q+\frac{1}{2}}^{l_{q}}Q_{q-1+\frac{1}{2}}^{l_{q-1}}\cdots Q_{1+\frac{1}{2}}^{l_{1}} Q_{\frac{1}{2}}^{l_{0}}, \quad
 G^{{\bf k}}=G_{p+\frac{1}{2}}^{k_{p}}G_{p-1+\frac{1}{2}}^{k_{p-1}}\cdots G_{1+\frac{1}{2}}^{k_{1}} G_{\frac{1}{2}}^{k_{0}}.$$
It is clear that $W$ is a $\mathfrak{b}$-module.

 For any $v\in W\setminus V$, let $supp(v)$ be the set of all $(\bf l,k)$ with $v_{\bf l,k}\ne 0$, and $deg(v)$ be the maximal element of $supp(v)$ with respect to the above  total order $\succ$.  Let $deg(v)=(\bf n,m)$.
If ${\bf m}\ne 0$, let $a=\min\{s\,|\,m_s\ne 0\}$. It is follows from $p+q+3=r$ that  $r-a-\frac{1}{2}\ge q+\frac{3}{2}$. Thus $Q_{r-a-\frac{1}{2}}V=0$. Since $a\le p$ and $p+q+3=r$, we obtain $2a-r+1\ne p-q$. It follows from \eqref{=s} that
 $$[Q_{r-a-\frac{1}{2}}, G_{a+\frac{1}{2}}]w=(2L_r+(2a-r+1)T_r)w\ne 0,\quad \forall 0\ne  w\in V.$$
 Then by straight calculations, one can see that $deg(Q_{r-a-\frac{1}{2}}v)=({\bf n, m}-\epsilon_a)$. If ${\bf m}=0$, let $b=\min\{s\,|\,n_s\ne 0\}$. Note that $r-b-\frac{1}{2}\ge p+\frac{3}{2}$ and $G_{r-b-\frac{1}{2}}V=0$.
  Since $b\le q$ and $p+q+3=r$, we obtain  $r-2b -1\ne p-q$. Thus from \eqref{=s} we have
 $$[G_{r-b-\frac{1}{2}}, Q_{b+\frac{1}{2}}]w=(2L_r+(r-2b-1)T_r)w\ne 0,\quad \forall 0\ne  w\in V.$$ Then one can see that $deg(G_{r-b-\frac{1}{2}}v)=({\bf n}-\epsilon_b,{\bf m})$. Hence,  $W$ is a simple $\mathfrak{b}$-module. Since  $p\ge 0,q\ge 0$, we have $p<r$ and $q<r$.
Then, one can further observe that
\begin{equation}\label{=sirr}\left\{\begin{array}{l}
 L_r, T_r \mbox{ act injectively on } W,\\
  L_r+\a T_r,\ \a\ne p-q \mbox{ act injectively on } W,\\
L_j(W)=T_j(W)=0,j>r,\\
 G_{k+\frac{1}{2}}(W)= Q_{k+\frac{1}{2}}(W)=0, k\ge r.\end{array}
\right.
\end{equation}
Now the simplicity of $Ind_\ell(V)$  follows from Lemma \ref{LTb-irr}.

Similarly,  we have the following Theorem.

\begin{theo} \label{LT-th-r+1}Let $V$ be a simple $\mathcal{L}_{0,p,q}$-module. Assume that there exist $r\in\Z_{\ge 1}, p,q\in\Z$ with $\max\{p,q\}<\max\{r,s\}$  such that $L_r,T_r, L_r+\a T_r, \a\in\Z$ act injectively on $V$, and $L_iV=T_iV=G_{k+\frac{1}{2}}V=Q_{l+\frac{1}{2}}V=0$ for all $i>r,k>p,l>q$.
If  $p+q+3=r+1$, then for all $\ell\in\C$, $Ind_\ell(V)$ is a simple $\mathcal{L}$-module.
\end{theo}

\begin{rema}\label{need-irr} If we drop the simplicity of $V$ in Theorem \ref{T-th} and  Theorems \ref{L-th}-\ref{LT-th-r+1}, one can also obtain a nonzero element in $V$ from any nonzero element in $Ind_\ell(V)\setminus V$.
\end{rema}


\section{Characterization of simple restricted modules}

In this section, we will give  precise characterizations of simple restricted modules over the $N=2$  Neveu-Schwarz algebra $\mathcal{L}$.

\begin{lemm}\label{N} Let $V$ be a nontrivial simple restricted $\mathcal{L}$-module. Then the following results hold.

{\rm (1)}  There exist $r,s,p,q\in\Z$ such that
$$N_{r,s,p,q}:=\{v\in V \,|\,L_iv=T_jv=G_{k+\frac{1}{2}}v=Q_{l+\frac{1}{2}}v=0, i>r,j>s,k>p,l>q\}\ne 0.
$$

{\rm (2)}  $r\ge -2$, $s\ge -1$, $p+q+3\ge -1$.

{\rm (3)} $r=-2$ implies  $s=-1,p=q=-2$.

{\rm (4)} $T_0(N_{r,s,p,q})\subset N_{r,s,p,q}$, $ C (N_{r,s,p,q})\subset N_{r,s,p,q}$.

{\rm (5)} There are  smallest integers $r,s,p,q$ such that $N_{r,s,p,q}\ne 0.$ From now on, we  use $ N_{r,s,p,q}$ to denote the following set
$$\{0\}\bigcup \left\{v\in V \ \bigg|\begin{array}{l}L_rv\ne 0,\quad T_sv\ne 0,\quad G_{p+\frac{1}{2}}v\ne 0,\quad  Q_{q+\frac{1}{2}}v\ne 0,\\
L_iv=T_jv=G_{k+\frac{1}{2}}v= Q_{l+\frac{1}{2}}v=0,\quad i>r,j>s,k>p,l>q.
\end{array}\right\}$$

 {\rm (6)} If $r\ge 0,s\ge 0$, then $L_r,T_s$ act injectively on  $N_{r,s,p,q}$.

\end{lemm}
{\it Proof.}  (1) Since $V$ is a nontrivial restricted $\mathcal{L}$-module, it follows  that $N_{r,s,p,q}\ne 0$ for sufficiently large $r,s,p,q$.

(2) Assume that $r<-2$, then there exists $0\ne v\in N_{r,s,p,q}$ such that $L_iv=T_jv=G_{k+\frac{1}{2}}v=Q_{l+\frac{1}{2}}v=0, i\ge -2,j>s,k>p,l>q$.  Applying $L_{-2}$ to
$$L_{-1}v,L_{0}v, T_{s+1}v, T_{s+2}v,G_{p+\frac{3}{2}}v,G_{p+\frac{5}{2}}v, Q_{q+\frac{3}{2}}v,Q_{q+\frac{5}{2}}v$$  repeatedly,
one can obtain that $L_iv=T_iv=G_{i+\frac{1}{2}}v=Q_{i+\frac{1}{2}}=0$ for all $i\in\Z$. Hence,
$\mathcal{L}v=0$ and $V$ is a trivial module. This is a contradiction. Thus $r\ge -2$.

Assume that $s<-1$, then there exists $0\ne v\in N_{r,s,p,q}$ such that $L_iv=T_jv=G_{k+\frac{1}{2}}v=Q_{l+\frac{1}{2}}v=0, i>r,j\ge -1,k>p,l>q$. Applying $T_{-1}$ to $G_{p+\frac{3}{2}}v$ and $Q_{q+\frac{3}{2}}v$ repeatedly,
one can obtain that $G_{k+\frac{1}{2}}v=Q_{k+\frac{1}{2}}v=0$ for all $k\in\Z$. Using this fact we  further deduce  that $\mathcal{L}v=0$. This is a contradiction. Thus $s\ge -1$.

Assume that $p+q+3<-1$.  Then we can find at least two integer pairs $(k,l)$  satisfying $k>p,l>q, k+l+1=-1$. Thus for $0\ne v\in N_{r,s,p,q}$, we have
$$[G_{k+\frac{1}{2}},Q_{l+\frac{1}{2}}]v=(2L_{-1}+(k-l)T_{-1})v=0,$$
from which we obtain that
$L_{-1}v=T_{-1}v=0.$ Then one can further deduce that $V$ is a trivial module and obtain a contradiction. Hence, $p+q+3\ge -1.$

(3) Assume that $r=-2$, there exists $0\ne v\in N_{r,s,p,q}$ such that $L_iv=T_jv=G_{k+\frac{1}{2}}v=Q_{l+\frac{1}{2}}v=0, i\ge -1,j>s,k>p,l>q$. From (2) we know that $s\ge -1$. If $p<-2$, apply $L_{-1}$ to $G_{p+\frac{3}{2}}v$ repeatedly, one can obtain that $G_{i+\frac{3}{2}}v=0$ for all $i\in\Z$. From this fact one can further deduce that $T_{-1}v=0$ and $V$ is a trivial module. Hence, $p\ge -2$. Similarly, one can show that $q\ge -2$. Now applying $L_{-1}$ to $T_{s+2}v,G_{p+\frac{5}{2}}v,Q_{q+\frac{5}{2}}v$ repeatedly,
one can obtain that $T_iv=G_{k+\frac{1}{2}}v=Q_{k+\frac{1}{2}}v=0$ for all $i\ge 0, k\ge -1$, which implies that $s\le -1,p,q\le -2$. Thus we must have $s=-1,p=q=-2$. Hence,  $r=-2$ implies  $s=-1,p=q=-2$.

(4) It is easy to verify that $T_0 v, C v\in N_{r,s,p,q}$ for every $v\in N_{r,s,p,q}$.

(5) follows from (2).

 (6) Since $r\ge 0, s\ge 0$, we observe that $L_iv, T_jv\in N_{r,s,p,q}$ for all $i\ge 0,j\ge 0, v\in N_{r,s,p,q}.$ Assume that $L_r$ does not act injectively on $N_{r,s,p,q}$. One can find $0\ne w\in N_{r,s,p,q}$ such that $L_iw=0, i\ge r$, which contradicts the minimality of $r$. Thus $L_r$ acts injectively  on $N_{r,s,p,q}$. Similarly, one can show that
$T_s$ acts injectively  on $N_{r,s,p,q}$.

Below  we fix the  following {\bf convention}: for  each $0\ne v\in N_{r,s,p,q}$,  we use $r,s,p,q$ to denote the  smallest integers such that $L_iv=0,i>r,T_jv=0,j>s,G_{k+\frac{1}{2}}v=0, k>p,Q_{l+\frac{1}{2}}v=0,l>q$, and  use $N$ to denote $N_{r,s,p,q}$ for brevity.  For $w:=G_{p+\frac{1}{2}}v$,
 we use $r',s',p',q'$ denote the  smallest integers such that $L_iw=0,i>r',T_jw=0,j>s',G_{k+\frac{1}{2}}w=0, k>p', Q_{l+\frac{1}{2}}w=0,l>q'$.

From now on, we retain the notation of Lemma \ref{N} and the convention established above.

\begin{lemm}\label{hw} Let $V$ be a nontrivial simple restricted $\mathcal{L}$-module. If $r,s\le 0, p,q\le -1$, then $V$ is  a highest weight module.
\end{lemm}

{\it Proof.} Since $V$ is simple, $C$ acts on $V$ as a  scalar. If $r,s<0$, it is clear that $V$ is  a highest weight module. Assume that $s=0$ or $r=0$. Note that $T_0$ (resp. $L_0$) acts injectively on $N$ if $s=0$ (resp. $r=0$).

Set $\mathfrak{T}=span\{L_0, T_0\}$.  Note that  maximal ideals of $U(\mathfrak{T})$ must be of the form $(L_0-a,T_0-b), a,b\in \C$. If $N$ is a  simple $\mathfrak{T}$-module, one easily see that $L_0,T_0$ act on $N$ as scalars, and $V$ is  a highest weight module.

If  $p+q+3=-1$,  there are at least two integer pairs $(k,l)$ satisfying $k>p,l>q, k+l+1=0$. Then we have
$$[G_{k+\frac{1}{2}},Q_{l+\frac{1}{2}}]v=2L_0v+(k-l)T_0v+\frac{C}{3}\left((k+\frac{1}{2})^2-\frac{1}{4}\right)v=0,\quad \forall v\in N.$$
This forces that $L_0,T_0$ act on $N$ as scalars, and we are done. In what follows, we assume that $p+q+3\ge 0.$

If $s=0, r<0$, then $N$ is a $\mathcal{L}_{0,p,q}$-module. If  there exist $c\in\Z_{\ge 1}$  and $0\ne v\in N$ such that $[Q_{c-\frac{1}{2}},G_{-c+\frac{1}{2}}]v=0$ or $[G_{c-\frac{1}{2}},Q_{-c+\frac{1}{2}}]v=0$, one can easily see that $T_0v=\l v$ for some $\l\in\C$. Since  $V$ is simple and generated by $v$, $V$ must be a highest weight module.  Assume that
\begin{equation}\label{0hw}[Q_{c-\frac{1}{2}},G_{-c+\frac{1}{2}}]v\ne 0 \mbox{ and } [G_{c-\frac{1}{2}},Q_{-c+\frac{1}{2}}]v\ne 0 \quad \forall c\in\Z_{\ge 1}, 0\ne v\in N.
\end{equation} Since  $V$ is simple and generated by $N$, there exists  the canonical map
 $$\pi: Ind^{\mathcal{L}}_{\mathcal{L}_{0,p,q}}(N)\rightarrow V,\quad \pi(1\otimes v)= v,\quad v\in N.$$
 Let $K=ker(\pi)$. It is clear that  $\pi$ is surjective and  $K\cap N=0$. If $K\ne 0$, then using \eqref{0hw} and arguments similar to those used in Theorem \ref{T-th} (cf. Lemma \ref{Tb-irr}), one can obtain  a nonzero element in $N$ from any nonzero element in $K$, which yields a contradiction. This proves $K=0$ and $V\cong Ind^{\mathcal{L}}_{\mathcal{L}_{0,p,q}}(N).$ By the property of induced modules, one can see that $N$ is a simple $\mathcal{L}_{0,p,q}$-module. Furthermore, $N$ is a simple $\mathfrak{T}$-module, and we are done.

If $r=0, s<0$, then $N$ is a $\mathcal{L}_{0,p,q}$-module. If  there exist $c\in\Z_{\ge 1}$  and $0\ne v\in N$ such that $[Q_{c-\frac{1}{2}},G_{-c+\frac{1}{2}}]v=0$ or $[G_{c-\frac{1}{2}},Q_{-c+\frac{1}{2}}]v=0$, one can easily see that $L_0v=\l v$ for some $\l\in\C$. It follows that  $V$ must be a highest weight module.  Assume that
$[Q_{c-\frac{1}{2}},G_{-c+\frac{1}{2}}]v\ne 0$ and $[G_{c-\frac{1}{2}},Q_{-c+\frac{1}{2}}]v\ne 0 $ for all $c\in\Z_{\ge 1}$ and all nonzero $v\in N$.
Then one can similarly show that  $V\cong Ind^{\mathcal{L}}_{\mathcal{L}_{0,p,q}}(N)$ and  $N$ is a simple $\mathcal{L}_{0,p,q}$-module. Furthermore, $N$ is a simple $\mathfrak{T}$-module, and we are done.

If $s=r=0,\, p+q+3=0$,   then $N$ is also a $\mathcal{L}_{0,p,q}$-module. Similarly, starting from any nonzero element of $\mathrm{Ind}^{\mathcal{L}}_{\mathcal{L}_{0,p,q}}(N)\setminus N$,
one can obtain a nonzero element in $N$. This yields $V\cong \mathrm{Ind}^{\mathcal{L}}_{\mathcal{L}_{0,p,q}}(N)$ and further shows that $N$ is a simple $\mathcal{L}_{0,p,q}$-module
(cf. Theorem \ref{LT-th} and Lemma \ref{LTb-irr}). Furthermore, $N$ is a  simple $\mathfrak{T}$-module, and we are done.

If $s=r=0,\, p+q+3=1$, then $p=q=-1$ and $N$ is a $\mathfrak{b}$-module.
By the PBW Theorem, each $v\in V\setminus N$ can be uniquely written in  the form
\begin{equation}
\sum\limits_{({\bf n',m',j',i'})\in \mathbb{S}_1\times \mathbb{S}_1\times \mathbb{S}_0\times \mathbb{S}_0}Q^{\bf n'}G^{\bf m'}T^{\bf j'}L^{\bf i'}v_{\bf n',m',j',i'},
\end{equation}
 where $v_{{\bf n',m',j',i'}}\in N.$ Let $supp(v)$ be the set of all $({\bf n',m',j',i'})$ with $v_{{\bf n',m',j',i'}}\ne 0$, and $deg(v)$ be the maximal element of $supp(v)$ with respect to the  order $\succ$ introduced in Section 2.  Let $deg(v)=({\bf n,m,j,i})$.
 Using similar arguments to those used in Section 3, one can prove  the following Claim.

{\bf Claim 3.}
(1) If ${\bf i}\ne 0$, let $a:=\min\{s\,|\,i_s\ne 0\}.$ Then $deg(T_{a}v)=({\bf n,m,j,i}-\epsilon_a).$

(2) If $ {\bf j}\ne 0$, let $b:=\min\{s\,|\,j_s\ne 0\}.$ Then  $deg(L_{b}v)=({\bf n,m,j}-\epsilon_b,{\bf i}).$

(3) If ${\bf m}\ne 0$, let $c:=\min\{s\,|\,m_s\ne 0\}.$  Assume further that $[Q_{c-\frac{1}{2}}, G_{-c+\frac{1}{2}}]v_{{\bf n,m,j,i}}\ne 0$. Then $deg(Q_{c-\frac{1}{2}}v)=({\bf n,m}-\epsilon_c,{\bf j,i}).$

(4) If $ {\bf n}\ne 0$, let $d:=\min\{s\,|\,n_s\ne 0\}.$  Assume further that $[G_{d-\frac{1}{2}}, Q_{-d+\frac{1}{2}}]v_{{\bf n,m,j,i}}\ne 0$. Then $deg(G_{d-\frac{1}{2}}v)=({\bf n}-\epsilon_d,{\bf m,j,i}).$

If  $[Q_{c-\frac{1}{2}}, G_{-c+\frac{1}{2}}]v_{{\bf n,m,j,i}}=0$ and $[G_{d-\frac{1}{2}}, Q_{-d+\frac{1}{2}}]v_{{\bf n,m,j,i}}=0$, one  easily sees that $L_0v_{{\bf n,m,j,i}}=\l v_{{\bf n,m,j,i}}, T_0v_{{\bf n,m,j,i}}=\mu v_{{\bf n,m,j,i}}$ for some $\l,\mu\in\C$, and we are done. Without loss of generality, we assume that $[Q_{c-\frac{1}{2}}, G_{-c+\frac{1}{2}}]v_{{\bf n,m,j,i}}=0$ and $[G_{d-\frac{1}{2}}, Q_{-d+\frac{1}{2}}]v_{{\bf n,m,j,i}}\ne 0$. From Claim 3 we can further assume that ${\bf i}= {\bf j}={\bf n}=0$.
If $c>1$, it is not difficult to see that $deg(T_1v)=({\bf 0},{\bf m}-\epsilon_c+\epsilon_{c-1},{\bf 0,0}).$  Combining this with Claim 3, one can similarly  show that
$V\cong Ind^{\mathcal{L}}_{\mathfrak{b}}(N)$ and $N$ is a simple $\mathfrak{b}$-module. Furthermore, $N$ is a simple $\mathfrak{T}$-module, and we are done.

If $c=1$, ${\bf w(n,m,j,i)}>\frac{1}{2}$ and there are $\a,\b\in \Z_{>0}$ with $\b>\a+1$ such that $m_1=\cdots=m_{\a-1}=m_\a=m_\b=1$ and $m_{\a+1}=\cdots=m_{\b-1}=0$, it is not difficult to see that $deg(T_1v)=({\bf 0},{\bf m}-\epsilon_\b+\epsilon_{\b-1},{\bf 0,0}).$ Then one can similarly show that $V\cong Ind^{\mathcal{L}}_{\mathfrak{b}}(N)$ and $N$ is a simple $\mathfrak{b}$-module. Thus we are done for this case.

If $c=1$, ${\bf w(n,m,j,i)}>\frac{1}{2}$ and ${\bf m}=\epsilon_1+\epsilon_2+\cdots +\epsilon_\a$ for some $\a\in\Z_{>1}$,  straightforward computation yields
\begin{eqnarray*}Q_{2-\frac{1}{2}}v&=&G^{({\bf 0},{\bf m}-\epsilon_2,{\bf 0}, {\bf 0})}\left((2L_0-3T_0+\frac{2C}{3}+\xi)v_{{\bf n,m,j,i}}\right)\\&+& \sum\limits_{({\bf n'',m'',j'',i''})\ne ({\bf 0},{\bf m}-\epsilon_2,{\bf 0}, {\bf 0})}Q^{\bf n''}G^{\bf m''}T^{\bf j''}L^{\bf i''}v_{\bf n'',m'',j'',i''},\end{eqnarray*}
where $\xi\in\C$ and $v_{\bf n'',m'',j'',i''}\in N$. If $(2L_0-3T_0+\frac{2C}{3}+\xi)v_{{\bf n,m,j,i}}=0$, using the relation $[Q_{c-\frac{1}{2}}, G_{-c+\frac{1}{2}}]v_{{\bf n,m,j,i}}=0$,  one  easily sees that $L_0v_{{\bf n,m,j,i}}=\l v_{{\bf n,m,j,i}}, T_0v_{{\bf n,m,j,i}}=\mu v_{{\bf n,m,j,i}}$ for some $\l,\mu\in\C$, and we are done.   If $(2L_0-3T_0+\frac{2C}{3}+\xi)v_{{\bf n,m,j,i}}\ne 0$, then  $deg(Q_{2-\frac{1}{2}}v)<deg(v)$. By arguments similar to those above, one can further show that $V$ is a highest weight module.

If $c=1$ and ${\bf w(n,m,j,i)}=\frac{1}{2}$, then we have $[Q_{\frac{1}{2}},G_{-\frac{1}{2}}]v_{0,\epsilon_1,0,0}=0$ and $v=G_{-\frac{1}{2}}v_{0,\epsilon_1,0,0}+a Q_{-\frac{1}{2}}v_{\epsilon_1,0,0,0}$ for some $a\in\C$.
Then  $Q_{-\frac{1}{2}}v=Q_{-\frac{1}{2}}(G_{-\frac{1}{2}}v_{0,\epsilon_1,0,0}).$ If $Q_{-\frac{1}{2}}(G_{-\frac{1}{2}}v_{0,\epsilon_1,0,0})=0$, set $w_1:=G_{-\frac{1}{2}}v_{0,\epsilon_1,0,0}$. Then  we  have
$$L_iw_1=T_iw_1=0, i\ge 1,\quad G_{k+\frac{1}{2}}w_1=0,k\ge -2,\quad Q_{l+\frac{1}{2}}w_1=0,l\ge -2.$$
Thus corresponding to $w_1$, we obtain $r',s'\le 0, p',q'\le -2.$  Therefore, this case reduces to the case $p+q+3\le -1$, for  which we have already shown that $V$ is  a highest weight module.
If $Q_{-\frac{1}{2}}(G_{-\frac{1}{2}}v_{0,\epsilon_1,0,0})\ne 0$, then $$G_{\frac{1}{2}}Q_{-\frac{1}{2}}(G_{-\frac{1}{2}}v_{0,\epsilon_1,0,0})=(2L_0+T_0)(G_{-\frac{1}{2}}v_{0,\epsilon_1,0,0})=G_{-\frac{1}{2}}(2L_0+T_0+2)v_{0,\epsilon_1,0,0}.$$
If $(2L_0+T_0+2)v_{0,\epsilon_1,0,0}=0$, using the relation $[Q_{\frac{1}{2}},G_{-\frac{1}{2}}]v_{0,\epsilon_1,0,0}=0$, one easily sees that $L_0v_{0,\epsilon_1,0,0}=\lambda v_{0,\epsilon_1,0,0}, T_0v_{0,\epsilon_1,0,0}=\mu v_{0,\epsilon_1,0,0}$ for some \(\lambda,\mu\in\mathbb{C}\), and we are done. If $(2L_0+T_0+2)v_{0,\epsilon_1,0,0}\ne 0$, set $w_2:=G_{-\frac{1}{2}}(2L_0+T_0+2)v_{0,\epsilon_1,0,0}$. Then  we have
$$L_iw_2=T_iw_2=0, i\ge 1,\quad G_{k+\frac{1}{2}}w_2=0,k\ge -1,\quad Q_{l+\frac{1}{2}}w_2=0,l\ge 0.$$
Thus corresponding to $w$, we obtain $r',s'\le 0, p'\le -2$ and $q'\le -1.$  Therefore, this case reduces to the case $p+q+3\le 0$, for  which we have already shown that $V$ is a highest weight module.

\begin{rema} (1) For $0\ne v\in N$, we know that $G_{p+\frac{1}{2}}v\ne 0$ and  $Q_{q+\frac{1}{2}}v\ne 0$. Since $G^2_{p+\frac{1}{2}}=Q^2_{p+\frac{1}{2}}=0$, we have $G_{p+\frac{1}{2}}(G_{p+\frac{1}{2}}v)=0$ and $Q_{p+\frac{1}{2}}(Q_{p+\frac{1}{2}}v)=0$. Thus, $G_{p+\frac{1}{2}}v, Q_{q+\frac{1}{2}}v \notin N$. Consequently, when we replace $v$ with $G_{p+\frac{1}{2}}v$ or $Q_{q+\frac{1}{2}}v$, $N$ is no longer the original space.

(2) The arguments employed to treat the case $p+q+3=1$ in Lemma \ref{hw} do not apply to Theorem \ref{LT-th-r+1}.
\end{rema}

\begin{prop}\label{hwrs0} Let $V$ be a nontrivial simple  restricted $\mathcal{L}$-module.  If $r\le 0, s\le 0$, Then $V$ is a highest weight module.
\end{prop}

{\it Proof.}  Let $w=G_{p+\frac{1}{2}}v$ for some $0\ne v\in N$. If $p+q+3>1$, then we have
\begin{equation}\label{hw-1}L_iw=T_iw=0,i>0,\quad  G_{k+\frac{1}{2}}w=0,k\ge p,\quad  Q_{l+\frac{3}{2}}w=0, l\ge q.
\end{equation}
Thus corresponding to  $w$, we have $r'\le 0, s'\le 0, p'\le p-1, q'\le q$ and $p'+q'+3<p+q+3$. Repeating the above procedure if necessary, we may assume that  $r\le 0, s\le 0, p+q+3\le 1.$

Without  loss of generality, assume that $p\ge q$. If $p<0$, $V$ is a highest weight module and we are done.  Assume  $p\ge 0$ and  set $w=G_{p+\frac{1}{2}}v$.  From Lemma \ref{N} (2),  we have the following three subcases.

{\bf subcase 1.} $p+q+3=-1$. In this subcase, we have $q\le -4,$  and  $G_{k+\frac{1}{2}}w=0, k\ge p,$ $ Q_{l+\frac{3}{2}}w=0, l\ge q+3.$ Thus $p'\le p-1,q'\le q+3<0$. Since $V$ is nontrival, from Lemma \ref{N}, we also have
$$G_{p-3+\frac{1}{2}}w\ne 0, \quad Q_{q+\frac{3}{2}}w\ne 0.$$
Thus  corresponding to  $w$, we obtain $p-3\le p'\le p-1, q'\le q+3$ and $-1\le p'+q'+3\le 1.$

If $p'=-1$, then $(p,q)=(0,-4),(1,-5),(2,-6)$. We obtain that $q'\le -1,-2,-3$. Combining this with \eqref{hw-1} and Lemma \ref{hw}, one can  see that $V$ is a highest weight module.

If $p'=-2$, then $(p,q)=(0,-4),(1,-5)$. We obtain that $q'\le -1,-2$, and $V$ is a highest weight module.

If $p'=-3$, then $(p,q)=(0,-4)$. We obtain that $q'\le -1$, and $V$ is a highest weight module.

If $p'\ge 0$,  we arrive at the case $r\le 0,s\le 0, p+q+3\le 1$ with strictly smaller $p$.

{\bf subcase 2.} $p+q+3=0$. In this subcase, we have $q\le -3,$  and $G_{k+\frac{1}{2}}w=0, k\ge p,$ $Q_{l+\frac{3}{2}}w=0, l\ge q+2.$ Thus $p'\le p-1,q'\le q+2<0$. Again from Lemma \ref{N}, we also have
$$G_{p-3+\frac{1}{2}}w\ne 0, \quad Q_{q+\frac{1}{2}}w\ne 0.$$
Thus corresponding to $w$, we obtain $p-3\le p'\le p-1, q'\le q+2$ and $-1\le p'+q'+3\le 1.$

If $p'=-1$, then $(p,q)=(0,-3),(1,-4),(2,-5)$. We obtain that $q'\le -1,-2,-3$, and $V$ is a highest weight module.

If $p'=-2$, then $(p,q)=(0,-3),(1,-4)$. We obtain that $q'\le -1,-2$, and $V$ is a highest weight module.

If $p'=-3$, then $(p,q)=(0,-3)$. We obtain that $q'\le -1$, and $V$ is a highest weight module.

If $p'\ge 0$,  we arrive at the case $r\le 0,s\le 0, p+q+3\le 1$ with strictly smaller $p$ .

{\bf subcase 3.} $p+q+3=1$. In this subcase, we have $q\le -2,$ and  $G_{k+\frac{1}{2}}w=0, k\ge p,$ $Q_{l+\frac{3}{2}}w=0, l\ge q+1.$ Again from Lemma \ref{N}, we also have
$$G_{p-3+\frac{1}{2}}w\ne 0, \quad Q_{q-\frac{1}{2}}w\ne 0.$$
Thus  corresponding to  $w$, we obtain $p-3\le p'\le p-1, q'\le q+1$ and $-1\le p'+q'+3\le 1.$

If $p'=-1$, then $(p,q)=(0,-2),(1,-3),(2,-4)$. We obtain that $q'\le -1,-2,-3$, and $V$ is a highest weight module.

If $p'=-2$, then $(p,q)=(0,-2),(1,-3)$. We obtain that $q'\le -1,-2$, and $V$ is a highest weight module.

If $p'=-3$, then $(p,q)=(0,-2)$. We obtain that $q'\le -1$, and $V$ is a highest weight module.

If $p'\ge 0$,  we arrive at  the case $r\le 0,s\le 0, p+q+3\le 1$ with  strictly smaller $p$.

Repeating  the above procedures if necessary, we finally conclude  that $V$ is a highest weight module.

\begin{coro}\label{pq0} Let $V$ be a nontrivial simple  restricted $\mathcal{L}$-module.  If $p+q+3\le 0$, then $V$ is a highest weight module.
\end{coro}

{\it Proof.} If $p+q+3\le 0$,  there are at least two integer pairs $(k,l)$ satisfying $k>p,l>q, k+l+1=i\ge 1$. Then we have
$$[G_{k+\frac{1}{2}},Q_{l+\frac{1}{2}}]v=(2L_i+(k-l)T_i)v=0,\quad \forall v\in N.$$
This forces that $L_iv=T_iv=0,i>0$. Thus $r\le 0,s\le 0$. It is  from Proposition \ref{hwrs0} that $V$ is a highest weight module.

\begin{prop}\label{pq<}  Let $V$ be a nontrivial simple  restricted $\mathcal{L}$-module, but not a highest weight module. If $r>0$ or $s>0$, we can assume that $p+q+3\le \max\{r+1,s+1\}.$
\end{prop}
{\it Proof.} Assume that  $p+q+3> \max\{r+1,s+1\}$. Set $w=G_{p+\frac{1}{2}}v$ for some $0\ne v\in N$. Then we have
\begin{eqnarray*}&&\left\{\begin{array}{ll} L_iw=0,i>0,& \mbox{ if }r\le 0,\\
 L_iw=0,i>r,& \mbox{ if }r> 0,
\end{array}\right.\quad \left\{\begin{array}{ll} T_jw=0,j>0,& \mbox{ if }s\le 0,\\
 T_jw=0,j>s,& \mbox{ if }s> 0,\end{array}\right.\\
&& G_{k+\frac{1}{2}}w=0,k\ge p,\quad  Q_{l+\frac{3}{2}}w=0, l\ge q.
\end{eqnarray*}
Thus corresponding to  $w$, we have  $r'\le \max\{0,r\},s'\le \max\{0,s'\}$, $p'+q'+3<p+q+3.$  Note that $V$ is not a highest weight module.
It is follows from Proposition \ref{hwrs0} and Corollary \ref{pq0} that $r'>0$ or $s'>0$, and  $p'+q'+3\ge 1$.
Repeating the above procedures if necessary, we can finally obtain that $r>0$ or $s>0$, and $p+q+3\le \max\{r+1,s+1\}.$

\begin{rema} In the proof of Proposition \ref{pq<}, one can also replace $v$ by $Q_{q+\frac{1}{2}}v$.
\end{rema}

\begin{lemm}\label{<pq} If $r>0$ or $s>0$, then $p+q+3\ge \max\{r,s\}.$ Furthermore, if $s<r, 0<r$, then $p+q+3>r$.
\end{lemm}

{\it Proof.} Assume that $p+q+3< \max\{r,s\}.$ Then there are at least two integer pairs $(k,l)$ satisfying  $k>p,l>q,k+l+1=\max\{r,s\}$. Thus for $0\ne v\in N$, we  have
$$[G_{k+\frac{1}{2}},Q_{l+\frac{1}{2}}]v=(2L_{\max\{r,s\}}+(k-l)T_{\max\{r,s\}})v=0,$$
which forces $L_{\max\{r,s\}}v=T_{\max\{r,s\}}v=0$. This is a  contradiction. Hence, $p+q+3\ge \max\{r,s\}.$

If $s<r, 0<r$, and $p+q+3\le r$, then one may choose $k,l\in\Z$ with $k\ge p, l\ge q, k+l+3=r$  such that
$$[G_{k+\frac{3}{2}},Q_{l+\frac{3}{2}}]v=(2L_{r}+(k-l)T_{r})v=2L_{r}v=0,\ \forall v\in N.$$
 This is a  contradiction. Hence,  $p+q+3>r$.

\begin{prop} Let $V$ be a nontrivial simple restricted $\mathcal{L}$-module.  Then $V$ must fall into one of the following cases:

{\rm (1)} $V$ is a highest weight module;

{\rm (2)} $0\le r<s,\ p+q+3=s,s+1$;

{\rm (3)} $0\le s<r,\ p+q+3=r+1$;

{\rm (4)} $0<r=s,\ p+q+3=s,s+1.$

In  cases {\rm (2)}-{\rm (4)}, we can further assume that $\max\{p,q\}< \max\{r,s\}$.
\end{prop}

{\it Proof.} Assume that  $V$ is not a highest weight module. Then  from Lemma \ref{N} and Proposition \ref{hwrs0}, we have all the remaining cases:

 (a) $r=-1, s>0$;
(b) $s=-1,r>0$; (c) $0\le r<s$; (d) $0\le s<r$; (e) $0< s=r$.

 We  will show that  cases (a)-(e)  can be reduced to  cases (2)-(4) with $\max\{p,q\}< \max\{r,s\}$ after  replacing $v$ by some  appropriate vectors in $V$.

  (a) $r=-1, s>0$. Assume that $p+q+3>s+1$ and  $w=G_{p+\frac{1}{2}}v$ for some $0\ne v\in N$. Then
$$ L_0w\ne 0,\ L_iw=0,i>0; \quad T_jw=0, j>s, \quad G_{k+\frac{1}{2}}w=0,k\ge p\,; \quad  Q_{l+\frac{3}{2}}w=0, l\ge q.$$  Thus corresponding to $w$, we have $r'=0, s'\le s,p'\le p-1, q'\le q,$ and $p'+q'+3<p+q+3$. Since  $V$ is not a highest weight module, from Proposition \ref{hwrs0} we  have $s'>0$. Thus this case   reduces to  case (c) with  strictly smaller $p+q+3$.

  (b) $s=-1,r>0$. Assume that $p+q+3>r+1$ and  set  $w=G_{p+\frac{1}{2}}v$ for some $0\ne v\in N$. Then
 $$ L_iw=0,i>r; \quad T_0w\ne 0, T_jw=0,j>0;\quad G_{k+\frac{1}{2}}w=0,k\ge p; \quad  Q_{l+\frac{3}{2}}w=0, l\ge q.$$
Thus corresponding to $w$, we have $r'\le r, s'=0,p'\le p-1,q'\le q,$ and $p'+q'+3<p+q+3$.  Again by Proposition \ref{hwrs0}, we  have $r'>0$. Thus this case  reduces to case (d) with   strictly smaller $p+q+3$.

  In cases (c)-(e), assume that $p+q+3>\max\{s+1,r+1\}$ and set  $w=G_{p+\frac{1}{2}}v$ for some $0\ne v\in N$. Then we have
 $$ L_iw=0,i>r; \quad  T_jw=0,j>s;\quad G_{k+\frac{1}{2}}w=0,k\ge p; \quad  Q_{l+\frac{3}{2}}w=0, l\ge q.$$
Thus corresponding to $w$, we have $r'\le r, s'\le s,p'\le p-1, q'\le q,$  and $p'+q'+3<p+q+3$. Again by Proposition \ref{hwrs0}, we  have $r'>0$ or $s'>0$. Thus these cases   reduces to  cases (a)-(e)  with  strictly smaller $p+q+3$.

From Corollary \ref{pq0}, we always have $p'+q'+3\ge 1$.  Repeating the above procedures if necessary, we can assume that $p+q+3\le \max\{s+1,r+1\}$ in all cases (a)-(e).
Combining this with  Lemma \ref{<pq},  cases (a)-(e)  reduce to the following five cases:
$$\begin{array}{ll} {\rm (i)}\,  r=-1, s>0, p+q+3=s, s+1;& {\rm (ii)}\,  s=-1,r>0, p+q+3=r+1; \\
{\rm (iii)}\,  0\le r<s, p+q+3=s,s+1; & {\rm (iv)}\, 0<r=s, p+q+3=s,s+1; \\
{\rm  (v)} \, 0\le s<r, p+q+3=r+1.&
\end{array}
$$

Next, we will show that cases (i) and (ii) can be reduced to cases (iii)-(v).

(i) $r=-1, s>0, p+q+3=s, s+1$.   Set $w=G_{p+\frac{1}{2}}v$ for some $0\ne v\in N$.  If $p+q+3=s$, then  from
$[G_{p+\frac{3}{2}},Q_{q+\frac{3}{2}}]v=(2L_s+(p-q)T_s)v=0,$
we have
\begin{equation}\label{a=p-q}
(2L_s+\a T_s)v\ne 0,\quad \forall \a\ne p-q.
\end{equation}
Thus
$$ L_0w\ne 0,\ L_iw=0,i>0; \quad T_jw=0, j>s,\quad G_{k+\frac{1}{2}}w=0,k\ge p\,; $$
$$ Q_{q+\frac{5}{2}}w\ne 0, \quad  Q_{l+\frac{1}{2}}w= 0,\ l\ge q+3.$$  Hence corresponding to $w$, $r'=0, s'\le s.$
If $s'=s$, then $G_{p-\frac{1}{2}}w\ne 0$, $p'+q'+3=s'+1$, and  this case  reduces to  case (iii) with $p+q+3=s+1.$  If $s'<s$, it is follows from   $V$ is not a highest weight module and Proposition \ref{hwrs0} that $s'>0$. From Lemma \ref{<pq} we also have $p'+q'+3\ge s'$. If  $p'+q'+3=s',s'+1$,  we arrive at case (iii). If $p'+q'+3>s'+1,$ we can proceed as in cases (c)-(e) to obtain cases (i)-(v) with strictly smaller $p+q+3$. Repeating the above argument, we can finally arrive at  cases (iii)-(v).

If  $p+q+3=s+1$, then we have
$$ L_0w\ne 0,\ L_iw=0,i>0; \quad T_jw=0, j>s,\quad G_{k+\frac{1}{2}}w=0,k\ge p\,; $$
$$ Q_{q+\frac{3}{2}}w\ne 0, \quad  Q_{l+\frac{1}{2}}w= 0,\ l\ge q+2.$$  Thus corresponding to $w$, $r'=0, s'\le s.$
If $s'=s$, then $G_{p-\frac{1}{2}}w\ne 0$, $p'+q'+3=s'$, and  this case reduces to  case (iii) with $p+q+3=s.$  If $s'<s$, then $s'>0$ and $p'+q'+3\ge s'$. If  $p'+q'+3=s',s'+1$,  we arrive at case (iii).
If $p'+q'+3>s'+1$, we can again proceed as in cases (c)-(e) to obtain cases (i)-(v) with strictly smaller $p+q+3$. Repeating the above argument, we can finally arrive at  cases (iii)-(v).

Similarly, one can show that case (ii)  can be reduced to cases (iii)-(v).

Finally, we will show that one can further assume that $\max\{p,q\}< \max\{r,s\}$ in cases (2)-(4). The technique we will use is similar to the above. We note that the above procedures are consistent with those below.

(i) $r=-1, s>0, p+q+3=s, s+1$. If  $\max\{p,q\}\ge s$, say $p\ge s$, then $q\le -2$. Set $w=G_{p+\frac{1}{2}}v$ for some $0\ne v\in N$. Then we have
$$ L_0w\ne 0,\ L_iw=0,i>0; \quad T_jw=0, j>s,\quad G_{k+\frac{1}{2}}w=0,k\ge p\,; $$
$$\left\{\begin{array}{rl} Q_{q+\frac{5}{2}}w\ne 0, \quad  Q_{l+\frac{1}{2}}w = 0,\ l\ge q+3,& \mbox{ if } p+q+3=s,\\
 Q_{l+\frac{1}{2}}w= 0,\ l\ge q+2,& \mbox{ if } p+q+3=s+1.\end{array}\right.$$
 Thus corresponding to $w$, $r'=0, s'\le s,p'\le p-1, q'\le q+2.$
If $s'=s$ or $s'<s, p'+q'+3=s',s'+1$, we reduce to case (iii) with strictly smaller $p$ and $q\le 0<s$. Otherwise, from Proposition \ref{hwrs0} and Lemma \ref{<pq}, we have $0<s'<s, p'+q'+3>s'+1$ and $q'\le 0<p'$. Then we can proceed as in  cases (a)-(e) to obtain cases  (iii)-(v) with strictly smaller $p+q+3$ and $p$. Repeating the above argument, we can finally arrive at   cases (iii)-(v) with $\max\{p,q\}< \max\{r,s\}.$

(ii) $s=-1,r>0, p+q+3=r+1$. If  $\max\{p,q\}\ge r$, say $p\ge r$, then $q\le -2$. Set $w=G_{p+\frac{1}{2}}v$ for some $0\ne v\in N$. Then we have
$$ T_0w\ne 0,\ T_iw=0,i>0; \quad L_jw=0, j>r,\quad G_{k+\frac{1}{2}}w=0,k\ge p,\quad Q_{l+\frac{1}{2}}w= 0,l\ge q+2.$$
Thus corresponding to $w$, $s'=0, r'\le r,p'\le p-1, q'\le q+1.$
If $r'=r$ or $r'<r, p'+q'+3=r'+1$, we arrive at case (v) with strictly smaller $p$ and $q<0<r$. Otherwise, from Proposition \ref{hwrs0} and Lemma \ref{<pq}, we have $0<r'<r, p'+q'+3>r'+1$ and $q'<0<p'$. Then we can  again  proceed as in  cases (a)-(e) to obtain cases  (iii)-(v) with strictly smaller $p+q+3$ and $p$. Repeating the above argument, we can finally arrive at  cases (iii)-(v) with $\max\{p,q\}< \max\{r,s\}.$

(iii) $0\le r<s, p+q+3=s,s+1$. If  $\max\{p,q\}\ge s$, say $p\ge s$, then $q\le -2$. Set $w=G_{p+\frac{1}{2}}v$ for some $0\ne v\in N$. Then we have
$$ L_iw=0,i>r; \quad T_jw=0, j>s,\quad G_{k+\frac{1}{2}}w=0,k\ge p\,; $$
$$\left\{\begin{array}{rl} Q_{q+\frac{5}{2}}w\ne 0, \quad  Q_{l+\frac{1}{2}}w = 0,\ l\ge q+3,& \mbox{ if } p+q+3=s,\\
 Q_{l+\frac{1}{2}}w= 0,\ l\ge q+2,& \mbox{ if } p+q+3=s+1.\end{array}\right.$$
 Thus corresponding to $w$, $r'\le r, s'\le s,p'\le p-1, q'\le q+2\le 0.$
If $r'>\max\{s',0\}, p'+q'+3=r'+1$, we arrive at case (v) with strictly smaller $p$ and $q'\le 0$; if $s'>\max\{r',0\}, p'+q'+3=s',s'+1$, we arrive at  case (iii) with strictly smaller $p$ and $q'\le 0$; if $r'=s'>0, p'+q'+3=s',s'+1$, we arrive at case (iv) with strictly smaller $p$ and $q'\le 0$. Otherwise, from Proposition \ref{hwrs0} and Lemma \ref{<pq}, we have $r'>\max\{s',0\}, p'+q'+3>r'+1$; $s'>\max\{r',0\}, p'+q'+3>s'+1$, or $r'=s'>0, p'+q'+3>s'+1$. Then we can again proceed  as in  cases (a)-(e) to obtain cases  (iii)-(v) with strictly smaller $p+q+3$ and $p$. Repeating the above argument, we can finally arrive at  cases (iii)-(v) with $\max\{p,q\}< \max\{r,s\}.$

For cases (iv)  and  (v), using the similar argument as in case (iii), we can finally arrive at  cases (iii)-(v) with $\max\{p,q\}< \max\{r,s\}.$


We are now ready to state one of our main results in this section.

\begin{theo}\label{MainTh} Let $V$ be a nontrivial simple restricted  $\mathcal{L}$-module, but  not a highest weight module.
Recall $$N=\{v\in V \,|\,L_iv=T_jv=G_{k+\frac{1}{2}}v=Q_{l+\frac{1}{2}}v=0, i>r,j>s,k>p,l>q\}\ne 0,$$ and  $L_r,T_s$ act injectively on $N$. Then the following results hold.

{\rm (1)} If $0\le s<r,\, p+q+3=r+1,\ p,q<r,$ then $N$ is a simple $\mathcal{L}_{0,p,q}$-module, and $V\cong Ind^{\mathcal{L}}_{\mathcal{L}_{0,p,q}}(N).$

{\rm (2)} If  $0<r=s, \,p+q+3=r, \ p,q<r$, then $N$ is a simple $\mathcal{L}_{0,p,q}$-module, and $V\cong Ind^{\mathcal{L}}_{\mathcal{L}_{0,p,q}}(N).$

{\rm (3)}  If  $0<r=s, \,p+q+3=r+1,\ p,q<r,$  and  $2L_r+\a T_r$ with $\a\in\Z$ acts injectively on $N$, then $N$ is a simple $\mathcal{L}_{0,p,q}$-module, and $V\cong Ind^{\mathcal{L}}_{\mathcal{L}_{0,p,q}}(N).$

{\rm (4)}  If  $0\le r<s,\, p+q+3=s,s+1, \ p,q<s$, then $N$ is a simple $\mathcal{L}_{s-r,p,q}$-module, and $V\cong Ind^{\mathcal{L}}_{\mathcal{L}_{s-r,p,q}}(N).$
\end{theo}
{\it Proof.} We prove (1) as an example, as the others can be proved similarly.

 If   $0\le s<r, p+q+3=r+1,$ it is clear that $N$ is a  $\mathcal{L}_{0,p,q}$-module. Since  $V$ is simple and generated by $N$, there exists  the canonical map
 $$\pi: Ind^{\mathcal{L}}_{\mathcal{L}_{0,p,q}}(N)\rightarrow V,\quad \pi(1\otimes v)= v,\quad v\in N.$$
Clearly, $\pi$ is surjective.  Next we only need to show that $\pi$ is also injective.  Let $K=ker(\pi)$. Obviously, $K\cap N=0$.

If $K\ne 0$, then as in
the proof of Lemma \ref{Lb-irr} and Theorem \ref{L-th} (cf. Remark \ref{no-irr} and Remark \ref{need-irr}), we can obtain  a nonzero element in $N$ from any nonzero element in $K$, which yields a contradiction. This proves $K=0$ and $V\cong Ind^{\mathcal{L}}_{\mathcal{L}_{0,p,q}}(N).$ By the property of induced modules, one can see that $N$ is a simple $\mathcal{L}_{0,p,q}$-module.

For $t\in\Z_{\ge 1}$, let $$\mathcal{L}^{(t)}=\sum\limits_{n\ge t}(\C L_n \oplus \C T_n) \oplus \sum\limits_{k\ge t}
\C G_{k+\frac{1}{2}}\oplus \sum\limits_{l\ge t}\C G_{l+\frac{1}{2}}.$$
Clearly,  each $\mathcal{L}^{(t)}$ is a  subalgebra of $\mathcal{L}_{0,p,q}$ (resp. $\mathcal{L}_{d,p,q}$) for all $t\ge \max\{d,p,q\}$.

\begin{rema} For any $0\ne v\in N$, there exists $t\ge \max\{1,p,q\}$ (resp. $t\ge \max\{1,d,p,q\}$) such that $\mathcal{L}^{(t)}v=0$.  One can easily check that $\{v\in N~|~\mathcal{L}^{(t)}v=0\}$  is a $\mathcal{L}_{0,p,q}$-module (resp. $\mathcal{L}_{d,p,q}$-module).
Thus $\{v\in N~|~\mathcal{L}^{(t)}v=0\}=N$. Set $$Ann_{\mathcal{L}_{0,p,q}}(N)=\{x\in \mathcal{L}_{0,p,q}\mid x(N)=0\}\  (resp.\, Ann_{\mathcal{L}_{d,p,q}}(N)=\{x\in \mathcal{L}_{d,p,q}\mid x(N)=0\}).$$ Then $\mathcal{L}^{(t)}\subset Ann_{\mathcal{L}_{0,p,q}}(N)\ \,(resp.\ \mathcal{L}^{(t)}\subset Ann_{\mathcal{L}_{0,p,q}}(N))$. Let
$$\mathcal{L}_{0,p,q}^{(t)}= \mathcal{L}_{0,p,q}/Ann_{\mathcal{L}_{0,p,q}}(N)\ (resp. \,\mathcal{L}_{d,p,q}^{(t)}= \mathcal{L}_{d,p,q}/Ann_{\mathcal{L}_{d,p,q}}(N)).$$ Then
 $\mathcal{L}_{0,p,q}^{(t)}$ \ (resp. \,$\mathcal{L}_{d,p,q}^{(t)}$) is a finite-dimensional solvable Lie superalgebra. Furthermore,
$N$ can be regarded as simple $\mathcal{L}_{0,p,q}^{(t)}$-module (resp. $\mathcal{L}_{d,p,q}^{(t)}$-module).
\end{rema}

Next, we present several equivalent conditions for   simple restricted $\mathcal{L}$-modules.
\begin{theo}\label{local}
Let $V$ be a simple $\mathcal{L}$-module. Then the following conditions are equivalent:

{\rm (1)} $V$ is a restricted $\mathcal{L}$-module;

{\rm (2)} There exists $t\in\Z_{\ge 1}$ such that  $L_i,T_i$  act locally finitely on $V$ for all $i\ge t$;

{\rm (3)} There exists $t\in\Z_{\ge 1}$ such that  $L_i,T_i$ act locally nilpotently on $V$ for all $i\ge t$;

{\rm (4)} There exists $t\in\Z_{\ge 1}$ such that $V$ is a locally finite $\mathcal{L}^{(t)}$-module;

{\rm (5)} There exists $t\in\Z_{\ge 1}$ such that $V$ is a locally nilpotent $\mathcal{L}^{(t)}$-module.

\end{theo}
{\it Proof.} $(5)\Rightarrow (3)\Rightarrow (2)$ and  $(4)\Rightarrow (2)$ are clear.
Since $V$ is a simple $\mathcal{L}$-module, $V$ is generated by $N$ (cf. Lemma \ref{N} and Theorem \ref{MainTh}).  Thus, $(1)\Rightarrow (5)$ and  $(1)\Rightarrow (4)$ are  also clear.   So we only need to  prove $(2)\Rightarrow (1)$.

Suppose that $V$ is a simple $\mathcal{L}$-module and there exists $t\in\Z_{\ge 1}$ such that $L_i,T_i$  act locally finitely on $V$ for all $i\ge t$.  By Lemma 4.1 in \cite{G}, there exist
$0\ne v\in V$ and $s\in\Z_{\ge t}$ such that $L_iv=T_iv=0, i\ge s.$

{\bf Claim.}  $G_{m+\frac{1}{2}}v=Q_{m+\frac{1}{2}}v=0, m\ge 2s.$

Fix $k\ge 2s$ and consider the vector space
$$W_k:=\sum\limits_{n\in\Z_{\ge 0}}\C T^n_sG_{k+\frac{1}{2}}v=\sum\limits_{n\in\Z_{\ge 0}}\C G_{ns+k+\frac{1}{2}}v.$$
It follows from  $T_s$ acts locally finitely on $V$ that  $W_k$ is finite-dimensional. Therefore,
$$W:=\sum\limits_{i=k}^{s+k-1}W_i=\sum\limits_{i=k}^{s+k-1}\left(\sum\limits_{n\in\Z_{\ge 0}}\C G_{ns+i+\frac{1}{2}}v\right)=\sum\limits_{n\in\Z_{\ge 0}}\C G_{n+k+\frac{1}{2}}v$$
is also finite-dimensional.  If $W\ne 0$,
then we can choose a minimal $p\in\Z_{\ge 0}$ such that
\begin{equation}\label{e0}G_{m+\frac{1}{2}}v+a_1G_{m+1+\frac{1}{2}}v+\cdots +a_pG_{m+p+\frac{1}{2}}v=0,\end{equation}
for some $m\ge k$ and $a_i\in\C, i=1,\cdots,p$ with $a_p\ne 0$.
Applying $T_s$ and $L_s$ to \eqref{e0}, we have
\begin{equation}\label{contra}\left\{\begin{array}{l}
G_{s+m+\frac{1}{2}}v+a_1G_{s+m+1+\frac{1}{2}}v+\cdots +a_pG_{s+m+p+\frac{1}{2}}v=0,\\
b_0G_{s+m+\frac{1}{2}}v+b_1G_{s+m+1+\frac{1}{2}}v+\cdots +b_pG_{s+m+p+\frac{1}{2}}v=0,\end{array}
\right.
\end{equation}
where $b_0=\frac{s}{2}-(m+\frac{1}{2})$, $b_i=a_i\left(\frac{s}{2}-(m+i+\frac{1}{2})\right),i=1,\cdots,p.$  It  follows from  $m\ge k\ge 2s$ that $b_0\ne 0$. Observe that $b_0a_i-b_i=ia_i, i=1,\cdots,p.$ Then from \eqref{contra}, we obtain
$$a_1G_{s+m+1+\frac{1}{2}}v+2a_2G_{s+m+1+\frac{1}{2}}v+\cdots +pa_pG_{s+m+p+\frac{1}{2}}v=0.$$
To avoid  a contradiction, we must have $p=0$. Thus $W=0$ and  $G_{m+\frac{1}{2}}v=0, m\ge 2s.$

Similarly, one can show that $Q_{m+\frac{1}{2}}v=0, m\ge 2s.$    This proves the Claim.

Now there exists  $0\ne v\in V$ such that
$$L_iv=T_iv=G_{j+\frac{1}{2}}v=Q_{j+\frac{1}{2}}v=0,\quad i\ge s, j\ge 2s.$$

Note that $V$ is simple and $V=U(\mathcal{L})v$. By the PBW Theorem, one can easily see that $V$ is a restricted $\mathcal{L}$-module.

Finally, using the techniques developed in \cite{MNTZ}, we can simplify the conditions in Theorem \ref{local}.

\begin{theo}\label{last} Let $V$ be a simple $\mathcal{L}$-module. If there exists $p\in\Z_{\ge 1}$ such that $L_p$ acts locally finitely on $V$, then $V$ is a restricted module.
\end{theo}

{\it Proof.}  Since $L_p$ acts locally finitely on $V$, there exist $\xi\in\C$  and $0\ne v\in V$ such that
$L_pv=\xi v$. Then by Lemma 3.10 in \cite{MNTZ}, there exists $K_L\in\Z_{>p}$ such that
\begin{equation}\label{L}L_iv=0,\ i\ge K_L.
\end{equation}

Fix $j\in\{1,2,\ldots,p\}$, we want to show that there exists $N_j\in\Z_{\ge 1}$ such that
\begin{equation}\label{j}T_{j+np}v=0,\ n\ge N_j.
\end{equation}

 Since $dim(\sum\limits_{n\in\Z_{\ge 0}}\C L^n_pT_{j}v)<\infty$ and $(L_p-\xi)T_{j}v=-jT_{j+p}v,$
there exists the smallest $n_j\in\Z_{\ge 0}$ such that
$$T_{j+s_jp}v,\ T_{j+(s_j+1)p}v,\ \ldots,\ T_{j+(s_j+n_j)p}v$$
are linearly dependent for some $s_j\in\Z_{\ge 0}.$ Hence, there exists a polynomial
$$P_j(x)=a_0^j+a_1^j(x-\xi)+\cdots + a^j_{n_j}(x-\xi)^{n_j},\quad a_0^ja^j_{n_j}\ne 0$$
such that
\begin{equation}\label{poly}P_j(L_p)T_{j+s_jp}v=0.
\end{equation}
Applying $L_p-\xi$ to \eqref{poly} repeatedly, we can obtain that
\begin{equation}\label{poly+}P_j(L_p)T_{j+ap}v=0,\quad a\ge s_j.
\end{equation}

Denote by $\partial(P_j(x))$ the degree of $P_j(x)$. If $\partial(P_j(x))=0$, we are done. Now assume that $\partial(P_j(x))>0$.
 By considering $P_j(L_p)T_{j+2rp}v$ and $[L_{rp},P_j(L_p)T_{j+rp}]v$, where $rp>\max\{K_L, s_jp\}$, we deduce from \eqref{L} and \eqref{poly+} that
 \begin{eqnarray}\label{b}&&b_0T_{j+2rp}v+b_1T_{j+(2r+1)p}v+\cdots +b_{n_j}T_{j+(2r+n_j)p}v=0,\\
&&\label{c}c_0T_{j+2rp}v+c_1T_{j+(2r+1)p}v+\cdots +c_{n_j}T_{j+(2r+n_j)p}v=0,
 \end{eqnarray}
where
\begin{eqnarray*}&&b_0=a_0^j ,\quad b_i=(-1)^{i}a_i^j\prod_{l=0}^{i-1}\left(j+(2r+l)p\right),i=1,\ldots,n_j,\\
&&c_0=-(j+rp)a_0^j, \ \ c_i=(-1)^{i+1}a_i^j\prod_{l=0}^{i}\left(j+(r+l)p\right),i=1,\ldots,n_j.
\end{eqnarray*}
Note that $\lim\limits_{r\rightarrow \infty}\frac{c_{n_j}}{b_{n_j}}=\infty$. For a sufficiently large $r$, solving equations  \eqref{b} and \eqref{c}, one can obtain  a polynomial $Q^*_j(x)$  such that
$\partial(Q^*_j(x))<\partial(P_j(x))$ and $Q^*_j(L_p)T_{j+2rp}v=0$. Write $Q^*_j(x)=(x-\xi)^{m}Q_j(x)$, where $Q_j(\xi)\ne 0, m\in\Z_{\ge 0}$. Then we have
\begin{equation}\label{r+m}Q_j(L_p)T_{j+(2r+m)p}v=0.
\end{equation}  Applying $L_p-\xi$ to \eqref{r+m} repeatedly,
we obtain that
$$Q_j(L_p)T_{j+ap}v=0,\quad a\ge 2r+m.$$
If $\partial(Q_j(x))=0$, we are done. If $\partial(Q_j(x))>0$, iterate the above process. Consequently, there exists  $N_j\in \Z_{\ge 1}$ such that $\eqref{j}$ hold.

Now set $K_T=\max\{(N_1+1)p,(N_2+2)p, \ldots, (N_p+1)p\}$. Then we have
\begin{equation}\label{T}T_{i}v=0,\ i\ge K_T.
\end{equation}

 Similarly, one can show that there exist  $K_G, K_Q\in\Z_{>0}$ such that
 \begin{equation}\label{GQ}G_{i+\frac{1}{2}}v=0,\ i\ge K_G,\quad Q_{j+\frac{1}{2}}v=0,\ j\ge K_Q.
\end{equation}

Set $t=\max\{K_L,K_T,K_G,K_Q\}$. Then from \eqref{L}, \eqref{T} and \eqref{GQ}, we have
\begin{equation}\label{LTG}L_iv= T_iv=G_{i+\frac{1}{2}}v= Q_{i+\frac{1}{2}}v=0,\ i\ge t
\end{equation} Since $V$ is simple, $V=U(\mathcal{L})v$. From the PBW Theorem and \eqref{LTG}, it is straightforward to  see that  $V$ is a restricted $\mathcal{L}$-module.

 \vs{8pt}

\noindent {\bf Data availability}  No data was used for the research described in this article.

\end{document}